\documentclass[11pt]{article}

\usepackage{amsbsy}
\usepackage{latexsym}
\usepackage{amsmath}
\usepackage{amsfonts}
\usepackage{amssymb}
\usepackage{amsthm}
\usepackage{graphicx}
\usepackage{epsfig}
\newcommand{\be}{\begin{equation}}
\newcommand{\ee}{\end{equation}}
\newcommand{\bea}{\begin{eqnarray}}
\newcommand{\eea}{\end{eqnarray}}
\newcommand{\beas}{\begin{eqnarray*}}
\newcommand{\eeas}{\end{eqnarray*}}

\newcommand{\<}  {\langle}
\renewcommand{\>}{\rangle}
\newcommand{\tg}{{\rm t}}

\newcommand{\G}{\Gamma}
\newcommand{\bG}{{\partial\G}}

\newcommand{\R}  {{\mathrm {I\kern-.7ex R}}}

\newcommand{\N}{\mbox{\rm I\kern-.18em N}}
\newcommand{\dist}{\mathop{\rm dist}\nolimits}

\newcommand{\bcurlS}[1]{\mathop{{\rm\bf curl}_{#1}}\nolimits}
\newcommand{\curlS}[1]{\mathop{{\rm curl}_{#1}}\nolimits}

\newcommand{\bi}[1]{{\boldsymbol{#1}}}

\newcommand{\lag}{_{\rm lag}}
\newcommand{\mor}{_{\rm mor}}
\newcommand{\uh}{\underline{h}}

\newcommand{\CG}{{\cal G}}

\newcommand{\CT}{{\cal T}}

\newcommand{\bn}{{\bf n}}
\newcommand{\bt}{{\bf t}}

\newcommand{\bvarphi}{{\mbox{\boldmath $\varphi$}}}

\newtheorem{theorem}{Theorem}[section]
\newtheorem{lemma}{Lemma}[section]

\newtheorem{prop}{Proposition}[section]
\newtheorem{remark}{Remark}[section]

\title{Mortar Boundary Elements}

\author{Martin Healey
\thanks{Department of Mathematical Sciences, Brunel University,
        Uxbridge, West London UB8 3PH, UK.
        email: {\tt martin.healey@brunel.ac.uk}}
        \and
        Norbert Heuer
\thanks{Facultad de Matem\'aticas,
        Pontificia Universidad Cat\'olica de Chile,
        Casilla 306, Correo 22, Santiago, Chile.
        email: {\tt nheuer@mat.puc.cl}.
        Partially supported by CONICYT-Chile
        through FONDECYT project no. 1080044.}}

\begin{document}
\date{}
\maketitle

\bigskip
\begin{abstract}
We establish a mortar boundary element scheme for hypersingular boundary
integral equations representing elliptic boundary value problems in
three dimensions. We prove almost quasi-optimal convergence of the
scheme in broken Sobolev norms of order $1/2$. Sub-domain decompositions
can be geometrically non-conforming and meshes must be quasi-uniform
only on sub-domains. Numerical results confirm the theory.


\end{abstract}

%%%%%%%%%%%%%%%%%%%%%%%%%%%%%%%%%%%%%%%%%%%%%%%%%%%%%%%%%%%%%%%%%%%%%%%%%%%%%%%%
\section{Introduction and model problem}
\setcounter{equation}{0}

In the finite element framework, mortar methods are used to discretize a
given problem independently on sub-domains. It is a non-overlapping
domain decomposition method. Necessary continuity requirements
on interfaces of the sub-domains are implemented via Lagrangian multipliers.
The motivation is to facilitate the construction of finite element meshes on
complicated domains and to allow for parallelization. Bernardi, Maday and
Patera introduced this technique and gave first analyses in
\cite{BernardiMP_93_DDM,BernardiMP_94_NNA}. Later, geometrically non-conforming
sub-domain decompositions and problems in $\R^3$ have been studied by
Ben Belgacem and Maday \cite{BenBelgacemM_97_MEM,BenBelgacem_99_MFE}.
There is a large number of publications on mortar methods, all dealing
with the discretization of differential equations of different types and
with related numerical linear algebra. The first papers, just mentioned,
derive a priori error estimates in the framework of non-conforming methods
involving a Strang type estimate.

In this paper we establish a mortar setting for the boundary element method
(BEM) and prove almost quasi-optimal convergence for a model problem
involving the hypersingular operator of the Laplacian. The advantages
of this domain decomposition scheme (easier construction of meshes and
availability of parallel techniques) also apply to the BEM.
To be precise, we apply the mortar technique directly to the boundary
element discretization, not as a coupling procedure between boundary and
finite elements as in \cite{ChernovMS_08_hpM}.
The analysis of finite elements for the discretization of boundary integral
equations of the first kind goes back to N\'ed\'elec and Planchard
\cite{NedelecP_73_MVE}, and Hsiao and Wendland \cite{HsiaoW_77_FEM}.
Stephan \cite{Stephan_86_BIE} studied boundary elements for singular
problems on open surfaces. Hypersingular boundary integral equations are
well posed in fractional Sobolev spaces of order $1/2$ and
conforming Galerkin discretizations require continuous basis functions.
Due to the non-existence of a trace operator in these Sobolev spaces,
needed for the analysis of interface conditions,
mortar boundary elements give rise to a variational crime.
Indeed, it turns out that there is no well-defined continuous variational
formulation of the mortar setting for the BEM.
Instead we will analyze the discrete
mortar scheme as a non-conforming method for the original un-decomposed
integral equation. We follow the analysis presented in
\cite{BenBelgacem_99_MFE} where projection and extension operators
are used to bound the approximation error in the kernel space (of functions
satisfying the Lagrangian multiplier condition). Note that there is
a shorter presentation by Braess, Dahmen and Wiener \cite{BraessDW_99_MAM}
where the simpler argument \cite[Remark III.4.6]{Braess_97_FET}
is used to bound this error by a standard approximation error
(in un-restricted spaces). Nevertheless, in our case the Strang type
error estimate has a more complicated structure and it is not straightforward
to follow the argument \cite[Remark III.4.6]{Braess_97_FET}.

We will make use of some preliminary results in \cite{GaticaHH_BLM,HeuerS_CRB}.
In \cite{GaticaHH_BLM} we studied the discretization of hypersingular
operators on open surfaces using functions that vanish only in a discrete
weak sense on the boundary of the surface. Such functions in general
do not belong to the energy space of the operator and require a different
variational setting. This setting will be used also for the mortar
boundary elements. In \cite{HeuerS_CRB} this setting served to establish
(non-conforming) Crouzeix-Raviart boundary elements and to prove their
quasi-optimal convergence. Main tool in that paper is a discrete
fractional-order Poincar\'e-Friedrichs inequality. It serves to show
ellipticity of the principal bilinear form of the discrete scheme.
In this paper we generalize this inequality to the geometrically
non-conforming case, needed for general mortar decompositions.
Again, it is needed to prove (quasi-) ellipticity of the principal
bilinear form.
Our model problem is defined on an open flat surface $\G$ with polygonal
boundary. We prove that, up to logarithmical
terms, the mortar boundary element method converges quasi-optimally,
subject to a compatibility condition  of the boundary meshes and the meshes
on the interfaces for the Lagrangian multipliers. Here we rely on
the known Sobolev regularity of the exact solution leading to almost
$O(h^{1/2})$-convergence where $h$ is the maximum mesh size.
Our techniques are applicable also to polyhedral surfaces and include
meshes of shape-regular triangles and quadrilaterals.

An overview of this paper is as follows. In the rest of this section
we recall definitions of fractional order Sobolev norms and formulate
the model problem. In Section~\ref{sec_mortar} we define the mortar
scheme and present the main result (Theorem~\ref{thm_main}) establishing
almost quasi-optimal convergence of the mortar boundary element method.
Technical details and proofs are given in Section~\ref{sec_tech}.
In Section~\ref{sec_num} we present some numerical results that underline
the stated convergence of the mortar BEM.

%%%%%%%%%%%%%%%%%%%%%%%%%%%%%%%%%%%%%%%%%%%%%%%%%%%%%%%%%%%%%%%%%%%%%%%%%%%%%%%%

\bigskip
First let us briefly define the needed Sobolev spaces.
We consider standard Sobolev spaces where the following norms are used:
For a bounded domain $S\subset\R^n$ and $0<s<1$ we define
\[
   \|u\|^2_{H^s(S)}:=\|u\|^2_{L^2(S)} + |u|^2_{H^s(S)}
\]
with semi-norm
\[
    |u|_{H^s(S)} := 
    \Bigl(
    \int_S \int_S \frac{|u(x)-u(y)|^2}{|x-y|^{2s+n}} \,dx\,dy
    \Bigr)^{1/2}.
\]
For $0<s<1$ the space $\tilde H^s(S)$ is defined as the completion of
$C_0^\infty(S)$ under the norm
\[
   \|u\|_{\tilde H^s(S)}
   :=
   \Bigl(
   |u|^2_{H^s(S)}
   +
   \int_S \frac{u(x)^2}{(\dist(x,\partial S))^{2s}} \,dx
   \Bigr)^{1/2}.
\]
For $s\in (0,1/2)$, $\|\cdot\|_{\tilde H^s(S)}$ and $\|\cdot\|_{H^s(S)}$
are equivalent norms whereas for $s\in(1/2,1)$ there holds
$\tilde H^s(S) = H_0^s(S)$, the latter space being the completion
of $C_0^\infty(S)$ with norm in $H^s(S)$. Also we note that
functions from $\tilde H^s(S)$ are continuously extendable by zero onto a
larger domain. For details see, e.g., \cite{LionsMagenes,Grisvard_85_EPN}.
For $s>0$, the spaces $H^{-s}(S)$ and $\tilde H^{-s}(S)$ are the
dual spaces of $\tilde H^s(S)$ and $H^s(S)$, respectively.

Let $\G$ be a plane open surface with polygonal boundary.
In the following we will identify $\G$ with a domain in $\R^2$,
thus referring to sub-domains of $\G$ rather
than sub-surfaces. The boundary of $\G$ is denoted by $\bG$.

Our model problem is:
{\em For a given $f\in L^2(\G)$ find $u\in\tilde H^{1/2}(\G)$ such that}
\be \label{Wuf}
   Wu(x):=-\frac 1{4\pi}\frac{\partial}{\partial\bn_x}
   \int_\G u(y) \frac{\partial}{\partial\bn_y}\frac 1{|x-y|} \,dS_y
   = f(x),\quad x\in\G.
\ee
Here, $\bn$ is a normal unit vector on $\G$, e.g. $\bn=(0,0,1)^T$.
We note that $W$ maps $\tilde H^{1/2}(\G)$ continuously onto $H^{-1/2}(\G)$ 
(see \cite{Stephan_87_BIE}).
We have the following weak formulation of (\ref{Wuf}).
{\em Find $u\in\tilde H^{1/2}(\G)$ such that}
\be \label{VF}
	\<Wu,v\>_\G = \<f,v\>_\G \qquad \forall v\in\tilde H^{1/2}(\G). 
\ee
Here, $\<\cdot,\cdot\>_\G$ denotes the duality pairing between $H^{-1/2}(\G)$
and $\tilde H^{1/2}(\G)$. Throughout, this generic notation will be used for
other dualities as well, the domain mentioned by the index.

A standard boundary element method (BEM) for the approximate solution of
(\ref{VF}) is to select a piecewise polynomial subspace
$\tilde X_h\subset\tilde H^{1/2}(\G)$ and to define an approximant
$\tilde u_h\in\tilde X_h$ by
\[
   \<W\tilde u_h,v\>_\G = \<f,v\>_\G \qquad \forall v\in \tilde X_h. 
\]
Such a scheme is known to converge quasi-optimally in the
energy norm, cf. Remark~\ref{rem_error} below.
In the numerical section we will compare such a conforming approximation
with a mortar approximation, for the case where the meshes are globally
conforming.

%%%%%%%%%%%%%%%%%%%%%%%%%%%%%%%%%%%%%%%%%%%%%%%%%%%%%%%%%%%%%%%%%%%%%%%%%%%%%%%%
\section{Mortar method and main result} \label{sec_mortar}
\setcounter{equation}{0}

In this section we introduce the mortar boundary element method for the
approximate solution of the model problem (\ref{VF}). First we discuss
the decomposition of $\G$ into sub-domains. Then we introduce the discrete
approximation spaces. The main result of this paper is given at the end
of this section.

%%%%%%%%%%%%%%%%%%%%%%%%%%%%%%%%%%%%%%%%%%%%%%%%%%%%%%%%%%%%%%%%%%%%%%%%%%%%%%%%
\subsection{Sub-domain decomposition}

We consider a decomposition of $\G$ into non-intersecting sub-domains $\G_i$,
$i=1,\ldots, N$, giving rise to a coarse mesh
\[
   \CT := \{\G_1, \ldots, \G_N\}.
\]
For ease of presentation we assume that each $\G_i$ is either a
triangle or quadrilateral. More general decompositions into polygonal sub-domains
can be dealt with by further decomposing into triangles and quadrilaterals and
by considering conforming interface conditions on additional interfaces.
The mesh $\CT$ can be
non-conforming but must satisfy the assumption (A1) below.
The diameter of a sub-domain $\G_i$ is denoted by $H_i$, and
\(
   H   := \max_{i=1,\ldots, N} H_i.
\)
The interface between two neighboring sub-domains $\G_i$, $\G_j$
($i\not=j$, $\bar\G_i\cap\bar\G_j$ contains more than a point) is denoted by
$\gamma_{ij}$. For our analysis below we need the following assumption.

\medskip
\noindent{\bf (A1)}
Each interface $\gamma_{ij}$ consists of an entire edge of $\G_i$ or $\G_j$.

\medskip
The (relatively) open edges of a sub-domain $\G_i$ are $\gamma_i^j$,
$j=1,\ldots, m$.
Here, $m$ is a generic number
($m=3$ if $\G_i$ is a triangle and $m=4$ otherwise).
Using the symbol $\bG$ for the boundary of $\G$, and similarly $\bG_i$ for the
boundary of $\G_i$, the skeleton of the sub-domain decomposition is
\[
   \gamma := \cup_{i=1}^N \bG_i \setminus\bG.
\]
According to assumption (A1) the skeleton is covered by a set of
non-intersecting edges $\gamma_{ij}$. We number the edges like
$\gamma_1$, \ldots, $\gamma_L$, giving a decomposition of the skeleton like
\(
   \bar\gamma = \cup_{l=1}^L \bar\gamma_l.
\)
In the following we will denote this decomposition of the skeleton by
\[
   \tau := \{\gamma_1, \ldots, \gamma_L\}.
\]
We will refer to these edges as {\em interface edges}.
Each interface edge $\gamma_l$ is the interface between two sub-domains
$\G_i$, $\G_j$ and is an entire edge of one or both of them.
Given an integer $l$ ($1\le l\le L$) we denote by $l\lag$
(respectively, $l\mor$) the number of a sub-domain which has $\gamma_l$ as
an edge (respectively, the number of the other sub-domain),
\[
   \gamma_l = \gamma_{l\lag, l\mor}.
\]
As mentioned before, the selection of the index pair $(l\lag, l\mor)$ for
$l\in\{1, \ldots, L\}$ is not unique but will be fixed for a specific
sub-domain decomposition of $\G$.
Below, we will introduce a Lagrangian multiplier on the interfaces and on
$\gamma_l$ we will use a mesh related to the mesh on $\G_{l\lag}$.
The side of $\gamma_l$ stemming from $\G_{l\mor}$ is usually called
{\em mortar} side in the finite element literature and this explains our
notation. The side defining the Lagrangian multiplier
is often called {\em non-mortar} side.

Corresponding to the decomposition of $\G$ we will need the product
Sobolev space
\[
   H^s(\CT) := \prod_{K\in\CT} H^s(K) = \prod^{N}_{i=1} H^s(\G_i)
\]
with usual product norm.

%%%%%%%%%%%%%%%%%%%%%%%%%%%%%%%%%%%%%%%%%%%%%%%%%%%%%%%%%%%%%%%%%%%%%%%%%%%%%%%%
\subsection{Meshes and discrete spaces}

On each of the sub-domains $\G_i$ ($i\in\{1, \ldots, N\}$) we consider
a (sequence of) regular, quasi-uniform meshes $\CT_i$ consisting of shape
regular triangles or quadrilaterals, $\bar\G_i=\cup_{T\in\CT_i}\bar T$.
The maximum diameter of the elements of $\CT_i$ is denoted by $h_i$ and
we use the symbols
\[
   \uh := \min_{i=1, \ldots, N} h_i, \qquad
   h   := \max_{i=1, \ldots, N} h_i.
\]
Throughout the paper we assume without loss of generality that $h<1$. This
makes the writing of logarithmic terms in $h$ easier.

In the case of $\G$ being a square, Figure~\ref{domainmesh}
shows a conforming sub-domain decomposition (a) and a
non-conforming sub-domain decomposition (b), both with globally
non-conforming meshes.

\begin{figure}[htb] 
\begin{center}
\includegraphics[width=0.8\textwidth]{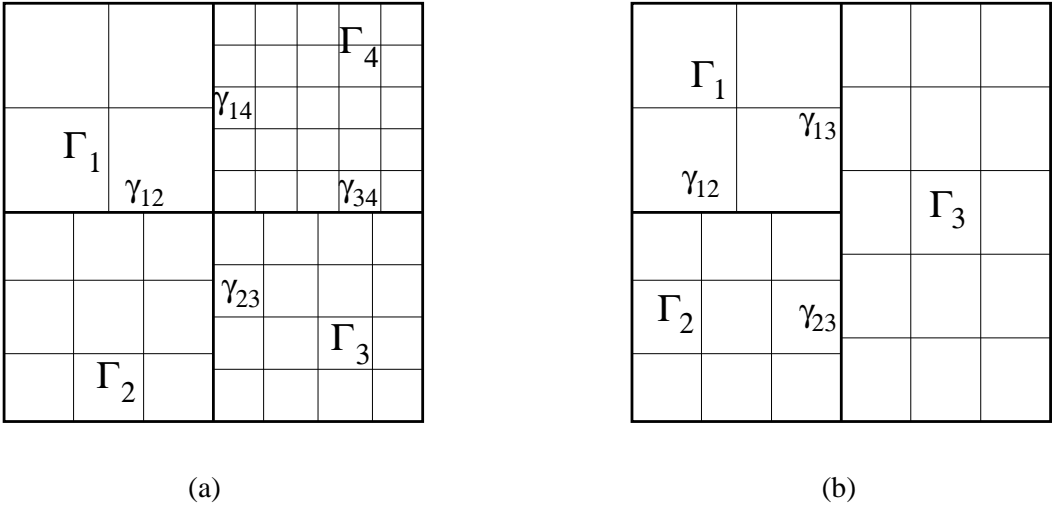}
\end{center}
\caption{Sub-domain decompositions with non-conforming meshes.}
\label{domainmesh}	
\end{figure}

Now we introduce discrete spaces on sub-domains consisting of piecewise
(bi)linear functions,
\[
   X_{h,i} := \{v\in C^0(\G_i);\; v|_T\ \mbox{is a polynomial of degree one}\
                \forall T\in\CT_i,\; v|_{\bG\cap\bG_i}=0\},\quad
   i=1,\ldots,N.
\]
The global discrete space on $\G$ is
\[
   X_h := \prod^N_{i=1} X_{h,i}.
\]
Note that functions $v\in X_h$ do satisfy the homogeneous boundary condition
along $\bG$ but are in general discontinuous across interfaces.
Therefore, $X_h$ is not a subspace of the energy space $\tilde H^{1/2}(\G)$.
Functions from different sub-domains will be coupled via a discrete
Lagrangian multiplier on the skeleton. To this end we introduce a mesh on
the skeleton $\gamma$ as follows.

On each interface edge $\gamma_l$ there is a trace mesh
$\CT_{l\lag}|_{\gamma_l}$ inherited from the mesh $\CT_{l\lag}$ on the
sub-domain $\G_{l\lag}$.
(We recall that by definition, $\gamma_l$ is an entire edge of the sub-domain
with number $l\lag$.)
This trace mesh is quasi-uniform with mesh width $h_{l\lag}$.
Now we introduce a new (coarser) quasi-uniform mesh $\CG_l$ on $\gamma_l$
in such a way that the following assumption is satisfied.

\medskip\noindent
{\bf (A2)} For any $l\in\{1, \ldots, L\}$ there holds:
           the mesh $\CG_l$ is a strict
           coarsening of the trace mesh $\CT_{l\lag}|_{\gamma_l}$.
           In particular, any interior node of $\CT_{l\lag}|_{\gamma_l}$
           together with its two neighboring elements (intervals) is covered
           by one element of $\CG_l$.

\medskip
The mesh width (length of longest element) of $\CG_l$ is denoted by $k_l$, and
\(
   k := \max_{l=1, \ldots, L} k_l.
\)
On each interface edge we define a space of piecewise constant functions,
\[
  M_{k,l}
  :=
  \{v\in L^2(\gamma_l);\; v|_J\ \mbox{is constant}\ \forall J\in\CG_l\},
  \quad l=1, \ldots, L.
\]
The space for the discrete Lagrangian multiplier then is
\[
	 M_k := \prod_{l=1}^L M_{k,l}.
\]

\noindent{\bf Notations.}
The symbols ``$\lesssim$'' and ``$\gtrsim$'' will be used in the usual sense.
In short, $a_h(v)\lesssim b_h(v)$ when there exists a constant $C>0$ independent
of the discretization parameter $h$ and the involved function $v$ such that
$a_h(v)\leq Cb_h(v)$ for any $v$ of the given set.
The double inequality $a_h(v)\lesssim b_h(v)\lesssim a_h(v)$ is simplified to
$a_h(v)\simeq b_h(v)$.
The generic constant $C$ above is usually also independent of appearing
fractional Sobolev indexes $\epsilon>0$, but this will be mentioned.
We note that these notations usually do not mean independence
of involved constants on the decomposition $\CT$ of $\G$.
In this paper we consider a generic decomposition $\CT$ which is fixed and
estimates will in general depend on $\CT$.

Throughout the paper we also will use the notation
$v_j$ for the restriction of a function $v$ to the sub-domain $\G_j$.

%%%%%%%%%%%%%%%%%%%%%%%%%%%%%%%%%%%%%%%%%%%%%%%%%%%%%%%%%%%%%%%%%%%%%%%%%%%%%%%%
\subsection{Setting of the mortar boundary element method and main result}

For the setup of the mortar boundary element method we need some operators.
We introduce the surface differential operators
\[
   \bcurlS{}\varphi
   :=
   \bigl(\partial_{x_2}\varphi, -\partial_{x_1}\varphi, 0\bigr),
   \qquad
   \curlS{}\bvarphi
   :=
   \partial_{x_1}\varphi_2 - \partial_{x_2}\varphi_1
   \quad\mbox{for}\quad
   \bvarphi=(\varphi_1,\varphi_2,\varphi_3).
\]
The definitions of the surface curl operators
are appropriate just for flat surfaces (as in our case) but can be extended
to open and closed Lipschitz surfaces, cf.~\cite{BuffaCS_02_THL,GaticaHH_BLM}.
We define corresponding piecewise differential operators
$\bcurlS{H}v$ and $\curlS{H}\phi$ by
\[
   \bcurlS{H} v := \sum^{N}_{i=1} \left(\bcurlS{\G_i} v_i\right)^0, \qquad  
   \curlS{H}\phi := \sum^{N}_{i=1} \left(\curlS{\G_i}\phi_i\right)^0.
\]
The notations $\bcurlS{\G_i}$ and $\curlS{\G_i}$ refer to the restrictions
of $\bcurlS{}$ and $\curlS{}$, respectively, onto $\G_i$,
and $(\cdot)^0$ indicates extension by zero to $\G$.
We made use of the notation introduced before,
$v_i=v|_{\G_i}$, $\phi_i=\phi|_{\G_i}$.
Furthermore, we need the single layer potential operator $V$ defined by
\[
   V\bvarphi(x):=\frac 1{4\pi}
   \int_\G \frac {\bvarphi(y)}{|x-y|} \,dS_y, \qquad
   \bvarphi\in (\tilde H^{-1/2}(\G))^3,\ x\in\G.
\]
For the formulation of the mortar boundary element method we define, for
sufficiently smooth functions $v$, $w$, $\mu$, the bilinear forms
$a(\cdot,\cdot)$ and $b(\cdot,\cdot)$ by
\beas
   a(v,w) &:=& \<V\bcurlS{H} v,\bcurlS{H} w\>_{\CT}
           :=  \sum_{i=1}^N \<V\bcurlS{H} v, \bcurlS{\G_i} w\>_{\G_i},
   \\
   b(v,\mu) &:=& \<[v],\mu\>_\tau
             := \sum_{l=1}^L \<[v], \mu\>_{\gamma_l}.
\eeas
Here, as mentioned before, for a domain $S\subset\G$ or an arc $S$,
$\<\cdot, \cdot\>_S$ denotes the $L^2(S)$-inner product and its extension by
duality, and $[v]$ is the jump of $v$ across $\gamma$, more precisely
\[
   [v]|_{\gamma_l} = v_{l\lag}|_{\gamma_l} - v_{l\mor}|_{\gamma_l},
   \quad l=1, \ldots, L.
\]
Of course, for sufficiently smooth functions $v$, $w$, $\mu$ there holds
\[
   a(v,w) = \<V\bcurlS{H} v,\bcurlS{H} w\>_{\G},\qquad
   b(v,\mu) = \<[v],\mu\>_\gamma.
\]
Note that we will use the introduced notations $\<\cdot,\cdot\>_{\CT}$
and $\<\cdot,\cdot\>_\tau$ for duality pairings of product spaces corresponding
to the given decompositions ($\CT$ and $\tau$).
We also define, for a sufficiently smooth function $v$, the linear form
\[
   F(v) := \sum^{N}_{i=1} \<f_i,v_i\>_{\G_i} = \<f, v\>_{\CT}
\]
where $f\in L^2(\G)$ is the function given in (\ref{Wuf}).
The mortar boundary element method for the approximate solution of (\ref{VF})
then reads:
{\em Find $u_h\in X_h$ and $\lambda_k\in M_k$ such that}
\bea \label{MBEM}
\begin{array}{lll}
	a(u_h,v) + b(v,\lambda_k) &=\ F(v) \qquad & \forall v\in X_h, \\
	b(u_h,\psi)   &=\ 0  \qquad & \forall\psi\in M_k.
	\end{array}		
\eea
This scheme is equivalent to: {\em Find $u_h\in V_h$ such that}
\[
   a(u_h,v)=F(v) \qquad\forall v\in V_h
\]
where
\be \label{def_Vh}
   V_h = \{v\in X_h;\; b(v,\psi) = 0 \quad\forall\psi\in M_k\}.
\ee
The main result of this paper is as follows.

\begin{theorem} \label{thm_main}
There exists a unique solution $(u_h, \lambda_k)$ of {\rm(\ref{MBEM})}.
Assume that the solution $u$ of {\rm(\ref{VF})} satisfies
$u\in\tilde H^{1/2+r}(\G)$ ($r\in (0,1/2]$). Then there holds
\[
   \|u-u_h\|_{H^{1/2}(\CT)}
   \lesssim
   (|\log\uh|^2 h^r + |\log\uh|^{3/2} k^r) \|u\|_{\tilde H^{1/2+r}(\G)}.
\]
For proportional mesh sizes $h$ and $k$ this means that
\[
   \|u-u_h\|_{H^{1/2}(\CT)}
   \lesssim
   |\log\uh|^2 h^r \|u\|_{\tilde H^{1/2+r}(\G)}.
\]
The appearing constants in the estimates above are independent of $h$ and $k$
provided that the assumptions on the meshes, in particular {\rm (A1)} and
{\rm (A2)}, are satisfied.
\end{theorem}

A proof of this theorem is given at the end of Section~\ref{sec_tech}.

\begin{remark} \label{rem_error}
As in {\rm\cite{GaticaHH_BLM}} we note that, in our case of an open surface
$\G$, the solution $u$ of {\rm(\ref{Wuf})} has strong corner and corner-edge
singularities which cannot be exactly described by standard Sobolev regularity.
It is well known that $u\in\tilde H^{s}(\G)$ for any $s<1$
(see, e.g., {\rm\cite{vonPetersdorffS_90_DEC}}) so that the error estimate by
Theorem~{\rm\ref{thm_main}} holds for any $r<1/2$. In general
$u\not\in H_0^1(\G)$ but a more specific error analysis for the conforming
BEM yields for quasi-uniform meshes the optimal error estimate
\[
   \|u-u_h\|_{\tilde H^{1/2}(\G)}
   \lesssim h^{1/2},
\]
see {\rm\cite{BespalovH_08_hpB}}.
The logarithmical perturbations in $\uh$ of our error estimate are due
to the non-conformity of the mortar method. They stem from the non-existence
of a trace operator within $H^{1/2}(\G)$ and from non-local properties of
the fractional order Sobolev norms (the difference between $\tilde H^{1/2}$ and
$H^{1/2}$-spaces).
\end{remark}

%%%%%%%%%%%%%%%%%%%%%%%%%%%%%%%%%%%%%%%%%%%%%%%%%%%%%%%%%%%%%%%%%%%%%%%%%%%%%%%%
\section{Technical details and proof of the main result} \label{sec_tech}
\setcounter{equation}{0}
\setcounter{figure}{0}

We start by citing some technical results
(Lemmas~\ref{Lemma_5_Heuer_01_ApS}-\ref{la_curl})
which are needed to deal with the fractional order Sobolev norms.
Afterwards we study a discrete Poincar\'e-Friedrichs inequality
(Proposition~\ref{prop_poincare}) which will be applied to prove
ellipticity of the bilinear form $a(\cdot,\cdot)$ on $V_h$.
Afterwards an integration-by-parts formula for the hypersingular
operator is recalled from \cite{GaticaHH_BLM} and adapted to our
situation of many sub-domains. Then, Lemma~\ref{la_IP} states the
well-posedness of integration by parts.
Lemmas~\ref{la_curl_pos}-\ref{la_b_cont} study requirements for
the Babu{\v s}ka-Brezzi theory and provide details for a
Strang-type error estimate which is given by Theorem~\ref{thm_Strang}.
Later, Lemmas~\ref{la_pi} and \ref{la_error_Vh} are needed to analyze the bound
of the Strang-type estimate and lead to Theorem~\ref{thm_error} which gives
a general a priori error estimate for the mortar BEM. The section
is finished by giving a proof of the main result (Theorem~\ref{thm_main}).

\begin{lemma} \label{Lemma_5_Heuer_01_ApS} {\rm\cite[Lemma 5]{Heuer_01_ApS}}
Let $R\subset\R^2$ be a Lipschitz domain. There exists $C>0$ such that
\[
   \|v\|_{\tilde H^s(R)} \leq \frac{C}{1/2-|s|} \|v\|_{H^s(R)}
   \quad\forall s\in(-1/2,1/2),\ \forall v\in H^s(R).
\]
\end{lemma}

\begin{lemma} \label{Lemma_6_Heuer_01_ApS}
Let $R\subset\R^2$ be a Lipschitz domain, and let $v$ be a piecewise linear
function defined on a quasi-uniform mesh on $R$ with mesh size $h<1$.
There exists a constant $C>0$ which is independent of $h$
(but may depend on $R$) such that there holds 
\[
   \|v\|_{\tilde H^{-1/2}(R)} \leq C |\log h|\, \|v\|_{H^{-1/2}(R)}.
\]
\end{lemma}

\begin{proof}
By \cite[Lemma 6]{Heuer_01_ApS} there holds for a piecewise polynomial
function of degree $p$ the estimate
\[
   \|v\|_{\tilde H^{-1/2}(R)} \leq C \log(\frac{p+1}{h})\|v\|_{H^{-1/2}(R)},
   \quad p\geq 0, h<1.
\]
Fixing $p$ gives the claimed bound.
The proof of \cite[Lemma 6]{Heuer_01_ApS} gives full details for
rectangular meshes. For triangular meshes
the proof applies as well by making use of Schmidt's inequality for
triangles, cf. \cite[Lemma 5.1]{Dorr_84_ATp}. Nevertheless, we are
considering only polynomials of low degrees where Schmidt's inequality
is not needed.
\end{proof}

\begin{lemma} \label{la_trace} {\rm\cite[Lemma 4.3]{GaticaHH_BLM}}
Let $R\subset\R^2$ be a bounded Lipschitz domain.
There exists $C>0$ such that, for any $\epsilon\in(0,1/2)$, there holds
\[
   \|v\|_{L^2(\partial R)}
   \le
   \frac C{\epsilon^{1/2}} \|v\|_{H^{1/2+\epsilon}(R)}
   \quad\forall v\in H^{1/2+\epsilon}(R).
\]
Here $\partial R$ is the boundary of $R$.
\end{lemma}

\begin{lemma} \label{la_curl}
For $S$ being one of the sub-domains $\G_i\in\CT$ or $\G$ there holds
\be \label{curl_semipos}
   |v|_{H^{1/2}(S)} \lesssim \|\bcurlS{S} v\|_{\bi{H}_\tg^{-1/2}(S)}
   \quad\forall v\in H^{1/2}(S).
\ee
The restriction of $\bcurlS{S}$ onto $\tilde H^{1/2}(S)$ is continuous,
\be \label{curl_tilde}
   \bcurlS{S}:\; \tilde H^{1/2}(S) \rightarrow \bi{\tilde H}^{-1/2}_\tg(S).
\ee
Moreover, there holds the continuity
\be \label{curl}
   \bcurlS{S}:\; H^{1/2+s}(S) \rightarrow \bi{H}^{-1/2+s}_\tg(S)
    \quad\forall s\in [0,1/2].
\ee
\end{lemma}

\begin{proof}
The bounds (\ref{curl_semipos}) and (\ref{curl_tilde}) are proved by Lemmas~4.1
and 2.2 in \cite{GaticaHH_BLM}, respectively.
By Lemma 2.1 in \cite{GaticaHH_BLM},
$\bcurlS{S}:\; H^{1/2}(S) \rightarrow \bi{H}^{-1/2}_\tg(S)$ is continuous,
and $\bcurlS{S}:\; H^1(S) \rightarrow \bi{L}^2_\tg(S)=\bi{H}^0_\tg(S)$
is continuous as well. Estimate (\ref{curl}) then follows by interpolation.
\end{proof}
The following result is a generalized version of a discrete
Poincar\'e-Friedrichs inequality in fractional order Sobolev spaces, cf.
Theorem~8 in \cite{HeuerS_CRB}.

\begin{prop} \label{prop_poincare}
There exists a constant $C>0$, independent of the decomposition $\CT$ as
long as sub-domains are shape-regular, such that for all
$\epsilon\in(0,1/2]$ there holds
\[
   \|v\|^2_{L^2(\G)}
   \le
   C \left( 
      \epsilon^{-1} |v|_{H^{1/2+\epsilon}(\CT)}^2
      +
      \sum_{l=1}^L |\gamma_l|^{-1-2\epsilon} (\int_{\gamma_l}[v]\,ds)^2 
   \right)
   \quad\forall v\in H^{1/2+\epsilon}(\CT),\; v|_{\partial\G}=0.
\]
Here, $|\gamma_l|$ denotes the length of $\gamma_l$.
\end{prop}

\begin{proof}
For the case of conforming decompositions $\tilde\CT$ of $\G$
into triangles, \cite[Theorem 8]{HeuerS_CRB} proves that there holds
\be \label{PF_conf}
   \|v\|^2_{L^2(\G)}
   \le
   C \left( 
      \epsilon^{-1} |v|_{H^{1/2+\epsilon}(\tilde\CT)}^2
      +
      \sum_{l=1}^L |\gamma_l|^{-1-2\epsilon} (\int_{\gamma_l}[v]\,ds)^2 
      +
      |\int_\G v\,dx|^2
   \right)
   \quad\forall v\in H^{1/2+\epsilon}(\tilde\CT).
\ee
It is easy to see that the mean zero term can be avoided
by assuming the homogeneous boundary condition for $v$.
To obtain the result for our non-conforming decomposition $\CT$ including
quadrilaterals we introduce further edges to
reduce quadrilateral sub-domains to triangles and to transform $\CT$ into a
conforming decomposition $\tilde\CT$.
By definition of the Sobolev-Slobodeckij semi-norm there holds
\[
   |v|_{H^{1/2+\epsilon}(\tilde\CT)} \le |v|_{H^{1/2+\epsilon}(\CT)}.
\]
We note that for new edges $\gamma'$ there holds $[v]|_{\gamma'}=0$ by the
trace theorem and the regularity $v\in H^{1/2+\epsilon}(T)$ for any $T\in\CT$.
The result then follows from (\ref{PF_conf}).
\end{proof}

Following \cite{GaticaHH_BLM} we now examine an integration-by-parts formula
for the hypersingular operator.
For a smooth scalar function $v$ and a smooth tangential vector field
$\bvarphi$, integration by parts gives
\[
   \<\bcurlS{\G_i} v,\bvarphi\>_{\G_i}
   =
   \<v,\curlS{\G_i}\bvarphi\>_{\G_i} - \<v,\bvarphi\cdot\bt_i\>_{\bG_i},
   \quad i=1, \ldots, N.
\]
Here, $\bt_i$ is the unit tangential vector on $\partial\G_i$ (oriented
mathematically positive when identifying $\G_i$ with a subset of $\R^2$
which is compatible with the identification of $\G$ as a subset of $\R^2$).
Applying this formula to $\bvarphi = (V\bcurlS{\G} u)|_{\G_i}$,
we obtain for smooth functions $v$ and $u$
\[
   \<\bt_i\cdot V\bcurlS{\G} u, v_i\>_{\bG_i}
   =
   \<\curlS{\G_i} V\bcurlS{\G} u, v_i\>_{\G_i}
   -
   \<V\bcurlS{\G} u,\bcurlS{\G_i} v_i\>_{\G_i}.
\]
Now we sum over $i$ and take into account that $t_i=-t_j$ on $\gamma_{ij}$.
Further we let $\gamma_0 := \bG$, use the convention for the jump
\(
   [v]|_{\gamma_0} = v|_{\gamma_0},
\)
denote by $\bt_0$ the unit tangential vector along $\bG$ (again mathematically
positive oriented) and let $0\lag:=0$ (remember the notation $l\lag$ and
$l\mor$ for the numbers of the Lagrangian multiplier side and mortar side of
$\gamma_l$, respectively).
This yields
\beas
   \sum_{l=0}^L \<\bt_{l\lag}\cdot V\bcurlS{\G} u, [v]\>_{\gamma_l}
   &=&
   \sum^N_{i=1} \<\curlS{\G_i} V\bcurlS{\G} u, v_i\>_{\G_i}
   -
   \sum^N_{i=1} \<V\bcurlS{\G} u, \bcurlS{\G_i} v_i\>_{\G_i}
   \\
   &=&
   \<\curlS{\G} V\bcurlS{\G} u, v\>_{\CT}
   -
   \<V\bcurlS{\G} u,\bcurlS{H} v\>_{\CT}
\eeas
for a piecewise (with respect to $\CT$) smooth function $v$ on $\G$ with
$v_i:=v|_{\G_i}$, as defined before.
%Here, more precisely the notation $\curlS{\G_i} V\bcurlS{\G} u$ means
%$\curlS{\G_i} (V\bcurlS{\G} u)|_{\G_i}$.
In the last step we used the fact that
\[
   \curlS{H} w = \curlS{\G} w\quad\forall w\in \bi{H}^{1/2}_\tg(\G),
\]
which holds by a density argument and the continuity of
$\curlS{\G}:\; \bi{H}^{1/2}_\tg(\G) \to H^{-1/2}(\G)$ as the adjoint
operator of
$\bcurlS{\G}:\; \tilde H^{1/2}(\G) \to \bi{\tilde H}^{-1/2}_\tg(\G)$,
cf. (\ref{curl_tilde}).

Now we use the relation
\[
   Wu= \curlS{\G} V \bcurlS{\G} u \quad (u\in \tilde H^{1/2}(\G)),
\]
see \cite{Maue_49_FAB,Nedelec_82_IEN} and \cite[Lemma 2.3]{GaticaHH_BLM}.
Then choosing a piecewise smooth function $v$ with $v|_\bG=0$ we obtain
\be \label{IP}
   \<\lambda, [v]\>_\tau
   =
   \sum_{l=1}^L \<\lambda, [v]\>_{\gamma_l}
   =
   \<Wu,v\>_{\CT} - \<V\bcurlS{\G} u, \bcurlS{H} v\>_{\CT}.
\ee
Here, $\lambda$ denotes our Lagrangian multiplier on the skeleton $\gamma$
defined by
\be \label{def_lambda}
   \lambda|_{\gamma_l}:= \bt_{l\lag}\cdot(V\bcurlS{\G} u)|_{\gamma_l},
   \quad l=1, \ldots, L.
\ee
Relation (\ref{IP}) does not extend to $v\in H^{1/2}(\CT)$ since the trace
of such a function $v$ onto $\gamma$ is not well defined.
However, there holds the following lemma.

\begin{lemma} \label{la_IP}
For $u\in\tilde H^{1/2}(\G)$ with $Wu = f\in L^2(\G)$,
\rm{(\ref{IP})} defines $\lambda\in\prod_{l=1}^L H^{-s}(\gamma_l)$
for any $s\in (0,1/2]$.
\end{lemma}

\begin{remark}
The above lemma can be extended to values of $s$ larger than $1/2$.
Though small values of $s$ represent the interesting cases, the limit $s=0$
being excluded.
Also, the condition on $f$ can be relaxed but excluding the case
$f\in H^{-1/2}(\G)$ which is the standard regularity using the mapping
properties of the hypersingular operator.
\end{remark}

\noindent{\bf Proof of Lemma~\ref{la_IP}.}
We must show that $\lambda$ defined by (\ref{IP}) is a bounded linear
functional on $\prod_{l=1}^L \tilde H^{s}(\gamma_l)$, the dual space of
$\prod_{l=1}^L H^{-s}(\gamma_l)$.

Let $v\in \prod_{l=1}^L \tilde H^{s}(\gamma_l)$ be given.
We continuously extend $v$ to an element
\(
   \tilde v\in H^{s+1/2}(\CT) 
\)
with $\tilde v=0$ on $\bG$ such that $[\tilde v]|_{\gamma_l}=v|_{\gamma_l}$.
(Simply extend $v$ on each interface edge $\gamma_l$ to a function in
$H^{s+1/2}(\G_{l\lag})$ vanishing on $\partial\G_{l\lag}\setminus\gamma_l$
and extend by zero to the rest of $\G$. Then sum up with respect to $l$.)
The definition of $\lambda$ is independent of the particular
extension $\tilde v$, see \cite{GaticaHH_BLM} for details in the case of one
sub-domain. Using a duality estimate we obtain from (\ref{IP})
\bea \label{pf1}
   \sum_{l=1}^L \<\lambda, [v]\>_{\gamma_l}
   &=&
   \<f,\tilde v\>_{\CT} - \<V\bcurlS{\G} u, \bcurlS{H} \tilde v\>_{\CT}
   \nonumber
   \\
   &\le&
   \|f\|_{L^2(\G)} \|\tilde v\|_{L^2(\G)}
   +
   \sum^N_{i=1} \|V\bcurlS{\G} u\|_{\bi{\tilde H}^{1/2-s}_\tg(\G_i)}
                \|\bcurlS{\G_i}\tilde v_i\|_{\bi{H}^{s-1/2}_\tg(\G_i)}.
\eea
Now, for $s\in (0,1/2]$ the norms in $\bi{\tilde H}^{1/2-s}_\tg(\G_i)$ and
$\bi{H}^{1/2-s}_\tg(\G_i)$ are equivalent (cf., e.g., \cite{LionsMagenes})
so that together with the mapping property of $V$ \cite{Costabel_88_BIO},
\[
   V:\; \bi{\tilde H}^{-1/2-s}_\tg(\G) \to \bi{H}^{1/2-s}_\tg(\G),
\]
and (\ref{curl_tilde}) we obtain
\[
   \sum^N_{i=1} \|V\bcurlS{\G} u\|_{\bi{\tilde H}^{1/2-s}_\tg(\G_i)}^2
   \lesssim
   \sum^N_{i=1} \|V\bcurlS{\G} u\|_{\bi{H}^{1/2-s}_\tg(\G_i)}^2
   \lesssim
   \|V\bcurlS{\G} u\|_{\bi{H}^{1/2-s}_\tg(\G)}^2
   \lesssim
   \|u\|_{\tilde H^{1/2}(\G)}^2.
\]
Here, the appearing constants are independent of $u$ but may depend on $s$.
Also, using (\ref{curl}) we are able to bound
(with constant independent of $\tilde v$)
\[
   \sum_{i=1}^N
   \|\bcurlS{\G_i}\tilde v_i\|_{\bi{H}^{s-1/2}_\tg(\G_i)}^2
   \lesssim
   \sum_{i=1}^N
   \|\tilde v_i\|_{H^{s+1/2}(\G_i)}^2
   =
   \|\tilde v\|_{H^{s+1/2}(\CT)}^2.
\]
Taking the last two estimates into account, (\ref{pf1}) proves that
\[
   \sum_{l=1}^L \<\lambda, [v]\>_{\gamma_l}
   \lesssim
   \left( \|f\|_{L^2(\G)}
          +
          \|u\|_{\tilde H^{1/2}(\G)}
   \right) 
   \|\tilde v\|_{H^{s+1/2}(\CT)}.
\]
Using the continuity of the extension (with constant independent of $v$)
\[
   \|\tilde v\|_{H^{s+1/2}(\CT)}^2
   \lesssim
   \sum_{l=1}^L \|v\|_{\tilde H^s(\gamma_l)}^2
\]
finishes the proof.
\qed

\begin{lemma} \label{la_curl_pos}
\[
   \|\bcurlS{H} v\|_{\bi{\tilde H}^{-1/2}_\tg(\G)}^2
   \gtrsim
   |\log\uh|^{-1} \|v\|_{H^{1/2}(\CT)}^2
   \quad\forall v\in V_h
\]
\end{lemma}

\begin{proof}
By (\ref{curl_semipos}) there holds
\be \label{pf2}
   \|\bcurlS{H} v\|_{\bi{\tilde H}^{-1/2}_\tg(\G)}^2
   \gtrsim
   \sum_{i=1}^N \|\bcurlS{\G_i} v_i\|_{\bi{H}^{-1/2}_\tg(\G_i)}^2
   \gtrsim
   \sum_{i=1}^N |v_i|_{H^{1/2}(\G_i)}^2
   =
   |v|_{H^{1/2}(\CT)}^2
   \quad\forall v\in H^{1/2}(\CT).
\ee
For $v\in V_h$ there holds $\int_{\gamma_l}[v]\,ds=0$ for any interface edge
$\gamma_l$ since by construction $M_k$ contains the piecewise constant function
which has the value $1$ on $\gamma_l$ and vanishes on $\gamma\setminus\gamma_l$.
For the definition of $V_h$ (the discrete kernel of $b(\cdot,\cdot)$) see
(\ref{def_Vh}). Therefore, Proposition~\ref{prop_poincare} proves that
\[
   \|v\|^2_{L^2(\G)}
   \lesssim
   \epsilon^{-1} |v|_{H^{1/2+\epsilon}(\CT)}^2
   \quad\forall v\in V_h.
\]
Here, the appearing constant is independent of $\epsilon\in (0,1/2]$.
Making use of the inverse property we bound
\[
   \sum^N_{i=1} |v|_{H^{1/2+\epsilon}(\G_i)}^2
   \lesssim
   \sum^N_{i=1} h_i^{-2\epsilon} |v|_{H^{1/2}(\G_i)}^2
   \quad\forall v\in V_h
\]
so that, with the previous estimate,
\be \label{pf3}
   \|v\|^2_{L^2(\G)}
   \lesssim
   \epsilon^{-1} \uh^{-2\epsilon} |v|_{H^{1/2}(\CT)}^2
   \quad\forall v\in V_h.
\ee
Selecting $\epsilon=|\log\uh|^{-1}$ (for $\uh$ being small enough)
and combining (\ref{pf3}) with (\ref{pf2}) proves the statement.
\end{proof}

\begin{lemma} \label{la_a_ell}
The bilinear form $a(\cdot,\cdot)$ is almost uniformly $V_h$-elliptic.
More precisely there hold the lower bounds
\[
   a(v,v) \gtrsim |\log\uh|^{-1} \|v\|_{H^{1/2}(\CT)}^2 \quad\forall v\in V_h
\]
and
\be \label{a_ell}
   a(v,v) \gtrsim
   |\log\uh|^{-1/2}
   \|v\|_{H^{1/2}(\CT)} \|\bcurlS{H} v\|_{\bi{\tilde H}^{-1/2}_\tg(\G)}
   \quad\forall v\in V_h.
\ee
\end{lemma}

\begin{proof}
First we note that for $v\in V_h\subset L^2(\G)$ there holds
$V\bcurlS{H} v\in \bi{L}^2_\tg(\G)$ so that
\[
   \<V\bcurlS{H} v, \bcurlS{H} v\>_{\CT}
   =
   \<V\bcurlS{H} v, \bcurlS{H} v\>_{\G}.
\]
Using the ellipticity of
$V:\; \bi{\tilde H}^{-1/2}_\tg(\G) \to \bi{H}^{1/2}_\tg(\G)$
and Lemma~\ref{la_curl_pos} we then obtain for $v\in V_h$
\[
   a(v,v) = \<V\bcurlS{H} v, \bcurlS{H} v\>_\G 
   \gtrsim
   \|\bcurlS{H} v\|_{\bi{\tilde H}^{-1/2}_\tg(\G)}^2
   \gtrsim  
   |\log\uh|^{-1} \|v\|^2_{H^{1/2}(\CT)},
\]
which is the first assertion.
The estimate (\ref{a_ell}) is obtained by bounding
$\|\bcurlS{H} v\|_{\bi{\tilde H}^{-1/2}_\tg(\G)}$ only once with the help of
Lemma~\ref{la_curl_pos}.
\end{proof}

\begin{lemma} \label{la_a_cont}
The bilinear form $a(\cdot,\cdot)$ is almost uniformly continuous on $X_h$.
More precisely there holds
\[
   a(v,w)
   \lesssim
   |\log\uh|^2\, \|v\|_{H^{1/2}(\CT)} \|w\|_{H^{1/2}(\CT)}
   \quad\forall v,w\in X_h.
\]
\end{lemma}

\begin{proof}
As in the proof of Lemma~\ref{la_a_ell} we note that for
$v$, $w\in X_h\subset L^2(\G)$ there holds
\[
   \<V\bcurlS{H} v, \bcurlS{H} w\>_{\CT}
   =
   \<V\bcurlS{H} v, \bcurlS{H} w\>_{\G}.
\]
Then, using the continuity of
$V:\; \bi{\tilde H}^{-1/2}_\tg(\G) \to \bi{H}^{1/2}_\tg(\G)$
and the estimate for fractional order Sobolev norms
$\|\cdot\|_{\bi{\tilde H}^{-1/2}_\tg(\G)}\lesssim
 \|\cdot\|_{\bi{\tilde H}^{-1/2}_\tg(\CT)}$, we obtain for $v,w\in X_h$
\bea \label{pf4}
   a(v,w) = \<V\bcurlS{H} v,\bcurlS{H} w\>_\G
   &\lesssim&
   \|\bcurlS{H} v\|_{\bi{\tilde H}^{-1/2}_\tg(\G)}
   \|\bcurlS{H} w\|_{\bi{\tilde H}^{-1/2}_\tg(\G)}
   \nonumber
   \\
   &\lesssim&
   \|\bcurlS{H} v\|_{\bi{\tilde H}^{-1/2}_\tg(\CT)}
   \|\bcurlS{H} w\|_{\bi{\tilde H}^{-1/2}_\tg(\CT)}.
\eea
Now making use of Lemma~\ref{Lemma_6_Heuer_01_ApS} and (\ref{curl}) we bound
\[
   \|\bcurlS{\G_i} v_i\|_{\bi{\tilde H}^{-1/2}_\tg(\G_i)}^2
   \lesssim
   |\log h_i|^2 \|\bcurlS{\G_i} v_i\|_{\bi{H}^{-1/2}_\tg(\G_i)}^2
   \lesssim
   |\log h_i|^2 \|v_i\|_{H^{1/2}(\G_i)}^2,
\]
giving
\[
   \|\bcurlS{H} v\|_{\bi{\tilde H}^{-1/2}_\tg(\CT)}
   \lesssim
   |\log\uh|\, \|v\|_{H^{1/2}(\CT)} \quad\forall v\in X_h.
\]
Combination with (\ref{pf4}) proves the statement.
\end{proof}

For the eventual error estimate we need the boundedness of the bilinear
form $a(\cdot,\cdot)$. However, Lemma~\ref{la_a_cont} is not applicable to
non-discrete functions. Instead we will use the next lemma.

\begin{lemma} \label{la_a_cont2}
Assume that $u\in\tilde H^{1/2+r}(\G)$ ($r>0$). Then there holds
\[
   a(u-v,w)
   \lesssim
   s^{-1} \|u-v\|_{H^{1/2+s}(\CT)} 
   \|\bcurlS{H} w\|_{\bi{\tilde H}^{-1/2}_\tg(\CT)}
   \quad\forall v,w\in X_h,\ \forall s\in (0,\min\{r,1/2\}].
\]
In particular, the appearing constant is independent of $s$.
\end{lemma}

\begin{proof}
First we note that (\ref{pf4}) holds also for continuous functions, so that
\[
   a(u-v,w)
   \lesssim
   \|{\bcurlS{H} (u - v)}\|_{\bi{\tilde H}^{-1/2}_\tg(\CT)}
   \|\bcurlS{H} w\|_{\bi{\tilde H}^{-1/2}_\tg(\CT)}
   \quad\forall v, w\in X_h.
\]
Using the continuous injection
$\bi{\tilde H}^{-1/2+s}_\tg(\G_i)\to \bi{\tilde H}^{-1/2}_\tg(\G_i)$
and Lemma~\ref{Lemma_5_Heuer_01_ApS} we bound for $i\in\{1, \ldots, N\}$
\[
   \|\bcurlS{\G_i}(u_i-v_i)\|_{\bi{\tilde H}^{-1/2}_\tg(\G_i)} 
   \le
   \|\bcurlS{\G_i}(u_i-v_i)\|_{\bi{\tilde H}^{-1/2+s}_\tg(\G_i)} 
   \lesssim
   s^{-1}
   \|\bcurlS{\G_i}(u_i-v_i)\|_{\bi{H}^{-1/2+s}_\tg(\G_i)}.
\]
The continuity of
$\bcurlS{\G_i}:\; H^{1/2+s}(\G_i)\to \bi{H}^{-1/2+s}_\tg(\G_i)$
for any $i\in\{1, \ldots, N\}$ by (\ref{curl}) finishes the proof.
\end{proof}

In order to analyze the error bound of the Strang-type estimate
by Theorem~\ref{thm_Strang} below, we need to extend functions from
interface edges to sub-domains.
This is also required to prove an inf-sup condition for the bilinear
form $b(\cdot,\cdot)$.

To this end let us define extension operators that extend piecewise
linear functions from interface edges to piecewise (bi)linear functions
on the corresponding Lagrangian sub-domain,
\be \label{def_El}
   E_l:\; X_{h,l\lag}|_{\bar\gamma_l} \to X_{h,l\lag},
   \quad l=1, \ldots, L.
\ee
Here, for $v\in X_{h,l\lag}|_{\bar\gamma_l}$ the extension $E_lv$ is
defined as the function of $X_{h,l\lag}$ that coincides with $v$ in the
nodes on $\bar\gamma_l$ stemming from the mesh $\CT_{l\lag}$ and is zero
in the remaining nodes of $\CT_{l\lag}$.

\begin{lemma} \label{la_E}
\[
   \|E_lv\|_{H^s(\G_{l\lag})}
   \lesssim
   h_{l\lag}^{1/2-s}\|v\|_{L^2(\gamma_l)}
   \quad\forall v\in X_{h,l\lag}|_{\bar\gamma_l},\ \forall s\in [0,1],
   \quad l=1, \ldots, L.
\]
In particular, the appearing constant is independent of $s$.
\end{lemma}

\begin{proof}
Using the equivalence of norms in finite dimensional spaces and scaling
properties of the $L^2$-norm one obtains, by taking into account the
construction of $E_l$,
\[
   \|E_lv\|_{L^2(\G_{l\lag})}^2
   \lesssim
   h_{l\lag} \|v\|_{L^2(\gamma_l)}^2
   \quad\forall v\in X_{h,l\lag}|_{\bar\gamma_l}.
\]
Analogously we find
\beas
   \|E_lv\|_{H^1(\G_{l\lag})}^2
   &=&
   \|E_lv\|_{L^2(\G_{l\lag})}^2 + |E_lv|_{H^1(\G_{l\lag})}^2
   \\
   &\lesssim&
   h_{l\lag} \|v\|_{L^2(\gamma_l)}^2
   +
   h_{l\lag}^{-1} \|v\|_{L^2(\gamma_l)}^2
   \lesssim
   h_{l\lag}^{-1} \|v\|_{H^1(\gamma_l)}^2
   \quad\forall v\in X_{h,l\lag}|_{\bar\gamma_l}.
\eeas
The result then follows by interpolation.
\end{proof}

\begin{lemma} \label{la_LBB}
The bilinear form $b(\cdot,\cdot)$ satisfies the discrete inf-sup condition
\[
   \exists\beta>0:\quad
   \sup_{v\in X_h\setminus\{0\}}
   \frac{b(v,\mu)}{\|v\|_{H^{1/2}(\CT)}} \ge \beta\, \|\mu\|_{L^2(\gamma)}
   \quad\forall \mu\in M_k.
\]
Here, the constant $\beta$ is independent of $h$ and $k$ subject to
the assumptions made on the meshes.
\end{lemma}

\begin{proof}
Let $\mu\in M_k$ be given. On each interface edge $\gamma_l$, $\mu$
is a piecewise constant function on $\CG_l$, a mesh that is coarser than
the trace mesh $\CT_{l\lag}|_{\bar\gamma_l}$ stemming from the Lagrangian
side $\G_{l\lag}$, cf. assumption (A2).
On $\gamma_l$ we construct a piecewise linear function
$w_l\in X_{h,{l\lag}}|_{\gamma_l}$ in the following way. For each element
$J\in\CG_l$, $w_l$ vanishes at the endpoints of $J$, coincides with $\mu$ at
one interior node of $J$ and is linearly interpolated elsewhere on $\gamma_l$.
See Figure~\ref{interfacemesh2} for an example where $\mu$ is represented
by the dashed line and $w_l$ by the solid line.
The bullets indicate the nodes of the mesh for the Lagrangian multiplier
and the dashes indicate additional nodes of the trace mesh
(from the Lagrangian multiplier side).

We then extend $w_l$ to $\tilde w_l$ in $X_{h}$ by first extending to
$E_lw_l\in X_{h,{l\lag}}$, cf.~(\ref{def_El}), and then further by zero
onto $\Gamma$. Eventually we define $v:=\sum_{l=1}^L \tilde w_l$.

Note that $\tilde w_l$ vanishes on all interface edges except $\gamma_l$.
The trace of $\tilde w_l$ onto $\gamma_l$ from $\G_{l\lag}$ equals $w_l$
whereas the trace coming from the other side $\G_{l\mor}$ vanishes. This
yields
\be \label{pf5}
   [v]=[\tilde w_l]=w_l \quad \mbox{on}\quad \gamma_l, \qquad l=1, \ldots, L.
\ee
By the construction of $w_l$ there holds, uniformly for $\mu\in M_k$,
\be \label{pf6}
   \|\mu\|_{L^2(\gamma_l)}^2
   \simeq
   \<w_l ,\mu\>_{\gamma_l} \simeq \|w_l\|_{L^2(\gamma_l)}^2,
   \quad l=1, \ldots, L.
\ee
Also, taking into account that each sub-domain $\G_i$ has a limited number of
(interface) edges, determined by the relation $l\in\{1,\ldots,L\}:\;l\lag=i$,
Lemma~\ref{la_E} yields
\bea \label{pf7}
   \|v\|_{H^{1/2}(\CT)}^2
   &=&
   \sum_{i=1}^N\;
   \Bigl\|\sum_{l\in\{1,\ldots,L\}:\; l\lag=i} E_lw_l
   \Bigr\|_{H^{1/2}(\G_i)}^2
   \nonumber
   \\
   &\lesssim&
   \sum_{i=1}^N
   \sum_{l\in\{1,\ldots,L\}:\; l\lag=i}
   \|E_lw_l\|_{H^{1/2}(\G_i)}^2
   =
   \sum_{l=1}^L
   \|E_lw_l\|_{H^{1/2}(\G_{l\lag)}}^2
   \lesssim
   \sum_{l=1}^L
   \|w_l\|_{L^2(\gamma_l)}^2
\eea
Now, using (\ref{pf5}), (\ref{pf6}) and (\ref{pf7}), we finish the proof
by bounding
\[
   b(v,\mu)
   =
   \sum_{l=1}^L \<[v], \mu\>_{\gamma_l}
   =
   \sum_{l=1}^L \<w_l ,\mu\>_{\gamma_l} 
   \simeq
   \|\mu\|_{L^2(\gamma)}
   \Bigl(\sum_{l=1}^L \|w_l\|_{L^2(\gamma_l)}^2\Bigr)^{1/2}
   \gtrsim
   \|\mu\|_{L^2(\gamma)} \|v\|_{H^{1/2}(\CT)}.
\]
\end{proof}

\begin{figure}[htb] 
\begin{center}
\includegraphics[width=0.7\textwidth]{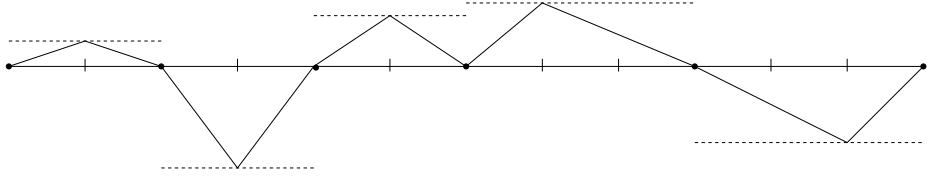}
\end{center}
\caption{Construction of $w_l$ in the proof of Lemma~\ref{la_LBB}.}
\label{interfacemesh2}			
\end{figure}

\begin{lemma} \label{la_bound_jump}
\[
   \|[v]\|_{L^2(\gamma)}^2
   \lesssim
   |\log\uh|\, \|v\|_{H^{1/2}(\CT)}^2
   \quad\forall v\in X_h
\]
\end{lemma}

\begin{proof}
By the triangle inequality and Lemma~\ref{la_trace} there holds
uniformly for $\epsilon\in (0,1/2)$
\beas
   \|[v]\|_{L^2(\gamma)}^2
   &\lesssim&
   \sum_{l=1}^L
   \left( \|v_{l\lag}\|_{L^2(\gamma_l)}^2
          + 
          \|v_{l\mor}\|_{L^2(\gamma_l)}^2
   \right) \\
   &\lesssim&
   \epsilon^{-1}
   \sum_{l=1}^L
   \left( \|v_{l\lag}\|_{H^{1/2+\epsilon}(\G_{l\lag})}^2
          + 
          \|v_{l\mor}\|_{H^{1/2+\epsilon}(\G_{l\mor})}^2
   \right)
   \lesssim
   \epsilon^{-1} \|v\|_{H^{1/2+\epsilon}(\CT)}^2
   \quad\forall v\in X_h.
\eeas
The inverse property, applied separately to $v_i=v|_{\G_i}$, yields
\[
   \|v\|_{H^{1/2+\epsilon}(\CT)}^2
   \lesssim
   \uh^{-2\epsilon} \|v\|_{H^{1/2}(\CT)}^2
   \quad\forall v\in X_h
\]
and selecting $\epsilon=|\log\uh|^{-1}$ (for $\uh$ being small enough)
finishes the proof.
\end{proof}

\begin{lemma} \label{la_b_cont}
The bilinear form $b(\cdot,\cdot)$ is almost uniformly discretely continuous,
in the sense that
\be \label{b_cont1}
   b(v, \mu)
   \lesssim
   |\log\uh|^{1/2}\, \|v\|_{H^{1/2}(\CT)} \|\mu\|_{L^2(\gamma)}
   \quad\forall v\in X_h, \forall \mu\in M_k.
\ee
Moreover, for given $\psi\in L^2(\gamma)$, there holds
\be \label{b_cont2}
   b(v, \psi)
   \lesssim
   |\log\uh|\, \inf_{\mu\in M_k} \|\psi - \mu\|_{L^2(\gamma)} 
   \|\bcurlS{H} v\|_{\bi{\tilde H}^{-1/2}_\tg(\G)}
   \quad\forall v\in V_h.  
\ee
\end{lemma}

\begin{proof}
There holds
\[
   b(v, \mu) = \sum_{l=1}^L \<[v], \mu\>_{\gamma_l}
   \le
   \|[v]\|_{L^2(\gamma)} \|\mu\|_{L^2(\gamma)}
   \quad\forall v\in X_h, \forall \mu\in M_k.
\]
Estimate (\ref{b_cont1}) follows with the help of Lemma~\ref{la_bound_jump}.
To prove (\ref{b_cont2}) we start as before and note that by
definition of $V_h$ there holds
$b(v,\mu)=0$ $\forall\mu\in M_k$, $\forall v\in V_h$.
Therefore, for any $\mu\in M_k$ and $v\in V_h$ we find that
\be \label{pf10}
   b(v, \psi)
   \le
   \|[v]\|_{L^2(\gamma)} \|\psi - \mu\|_{L^2(\gamma)}.
\ee
The proof of (\ref{b_cont2}) is finished by noting that combination of
Lemmas~\ref{la_bound_jump} and \ref{la_curl_pos} yields
\[
   \|[v]\|_{L^2(\gamma)}
   \lesssim
   |\log\uh|^{1/2} \|v\|_{H^{1/2}(\CT)}
   \lesssim
   |\log\uh|\, \|\bcurlS{H} v\|_{\bi{\tilde H}^{-1/2}_\tg(\G)}
   \quad\forall v\in V_h.
\]
\end{proof}

We are now ready to prove the following Strang-type error estimate.

\begin{theorem} \label{thm_Strang}
System {\rm(\ref{MBEM})} is uniquely solvable.
Let $u$ and $u_h$ be the solutions of {\rm(\ref{VF})} and {\rm(\ref{MBEM})},
respectively. Assuming that $u\in\tilde H^{1/2+r}(\G)$ ($r\in (0,1/2]$)
there holds
\[
   \|u-u_h\|_{H^{1/2}(\CT)}
   \lesssim
   |\log\uh|^{1/2}
   \left(
      s^{-1}
      \inf_{v\in V_h} \|u-v\|_{H^{1/2+s}(\CT)}
      + 
      \sup_{w\in V_h\setminus\{0\}} 
      \frac{a(u-u_h, w)}{\|\bcurlS{H} w\|_{\bi{\tilde H}^{-1/2}_\tg(\G)}}
   \right).
\]
uniformly for $s\in (0, \min\{1/2,r\}]$.
\end{theorem}

\begin{proof}
The existence and uniqueness of $(u_h,\lambda_k)\in X_h\times M_k$ 
follows from the Babu\v{s}ka-Brezzi theory.
Indeed, the bilinear form $a(\cdot, \cdot)$ is continuous on
$X_h$ by Lemma~\ref{la_a_cont} and $V_h$-elliptic by Lemma~\ref{la_a_ell},
and the bilinear form $b(\cdot,\cdot)$ is continuous on $X_h\times M_k$
by (\ref{b_cont1}) and satisfies a discrete inf-sup condition by
Lemma~\ref{la_LBB}.
The continuity and ellipticity bounds depend on $h$ but that does not
influence the unique solvability of the discrete scheme.

The error estimate is obtained by the usual steps.
Combining the triangle inequality, the non-standard
ellipticity and continuity properties of $a(\cdot, \cdot)$, cf.
(\ref{a_ell}) and Lemma~\ref{la_a_cont2}, we obtain for any $v\in V_h$
\beas
\lefteqn{
   \|u-u_h\|_{H^{1/2}(\CT)}
   \le
   \|u-v\|_{H^{1/2}(\CT)} + \|v-u_h\|_{H^{1/2}(\CT)}
}
   \\
   &\lesssim&
   \|u-v\|_{H^{1/2}(\CT)}
   +
   |\log\uh|^{1/2}
   \sup_{w\in V_h\setminus\{0\}} 
   \frac{a(v-u_h, w)}{\|\bcurlS{H} w\|_{\bi{\tilde H}^{-1/2}_\tg(\G)}}
   \\
   &\le&
   \|u-v\|_{H^{1/2}(\CT)}
   +
   |\log\uh|^{1/2}
   \left(
      \sup_{w\in V_h\setminus\{0\}} 
      \frac{a(v-u, w)}{\|\bcurlS{H} w\|_{\bi{\tilde H}^{-1/2}_\tg(\G)}}
      +
      \sup_{w\in V_h\setminus\{0\}} 
      \frac{a(u-u_h, w)}{\|\bcurlS{H} w\|_{\bi{\tilde H}^{-1/2}_\tg(\G)}}
   \right)
   \\
   &\lesssim&
   \|u-v\|_{H^{1/2}(\CT)}
   +
   s^{-1} |\log\uh|^{1/2} \|u-v\|_{H^{1/2+s}(\CT)}
   +
   |\log\uh|^{1/2}
   \sup_{w\in V_h\setminus\{ 0\}} 
   \frac{a(u-u_h, w)}{\|\bcurlS{H} w\|_{\bi{\tilde H}^{-1/2}_\tg(\G)}}.
\eeas
This proves the stated error bound.
\end{proof}

In order to analyze the upper bound provided by Theorem~\ref{thm_Strang}
we need, apart from the extension operators $E_l$ defined before,
projection operators $\pi_l$ acting on $L^2(\gamma_l)$
and mapping onto special continuous, piecewise linear functions on
$\gamma_l$, $l=1, \ldots, L$.
We recall that on each $\gamma_l$ we have two meshes: the trace
mesh $\CT_{l\lag}|_{\gamma_l}$ stemming from the mesh on the sub-domain
$\G_{l\lag}$ of the Lagrangian side, and the mesh $\CG_l$ for the Lagrangian
multiplier. For each element $J\in\CG_l$ we consider a hat function
$\phi_{l,J}$ that vanishes at the endpoints of $J$ and has the tip
at a node of $\CT_{l\lag}|_{\gamma_l}$ that is interior to $J$.
This choice is not unique if $J$ contains more than two elements
of the trace mesh. In that case we select an arbitrary but fixed node for
the definition of $\phi_{l,J}$. Using this notation we define
\be \label{def_pil}
   \pi_l:\; L^2(\gamma_l)
   \to\,
   {\rm span}\{\phi_{l,J};\; J\in\CG_l\} \subset X_{h,l\lag}|_{\bar\gamma_l},
   \quad l=1, \ldots, L,
\ee
such that the integral mean zero conditions
\[
   \<v - \pi_lv, 1\>_J = 0 \quad\forall J\in \CG_l,\; l=1, \ldots, L,
\]
hold. This operator satisfies the following properties.

\begin{lemma} \label{la_pi}
For any $v\in L^2(\gamma_l)$, $\pi_lv$ vanishes at the endpoints of
$\gamma_l$, $l=1, \ldots, L$, and there holds
\be \label{pi_orth}
   \<v - \pi_l v, \mu\>_{\gamma_l} = 0
   \quad\forall v\in L^2(\gamma_l),\; \forall\mu\in M_{k,l},
   \quad l=1, \ldots, L,
\ee
\be \label{pi_cont}
   \|\pi_lv\|_{L^2(\gamma_l)} \lesssim \|v\|_{L^2(\gamma_l)}
   \quad\forall v\in L^2(\gamma_l),\; l=1, \ldots, L.
\ee
\end{lemma}

\begin{proof}
For $l\in\{1, \ldots, L\}$ let $v\in L^2(\gamma_l)$ be given.
By definition of $\pi_l$, $\pi_lv$ vanishes at the endpoints of $\gamma_l$,
and the orthogonality (\ref{pi_orth}) follows by noting that any
$\mu\in M_{k,l}$ is constant on any $J\in\CG_l$.

To show (\ref{pi_cont}) let $J\in\CG_l$ be given. With $\phi_{l,J}$ being
the hat function defined previously (with height $1$) there holds
\[
   \pi_lv = \frac 2{|J|} (\int_J v\,ds)\; \phi_{l,J}
   \quad\mbox{on}\quad J
\]
so that
\[
   \|\pi_lv\|_{L^2(J)}^2
   =
   \frac 4{3|J|} (\int_Jv\,ds)^2
   \le
   \frac 43 \|v\|_{L^2(J)}^2.
\]
Summing over $J\in\CG_l$ finishes the proof.
\end{proof}

We are now ready to analyze the first term of the upper bound provided
by Theorem~\ref{thm_Strang}.

\begin{lemma} \label{la_error_Vh}
For $r\in (0,1/2]$ let $u\in H^{1/2+r}(\G)$. There holds
\[
   \inf_{v\in V_h} \|u - v\|_{H^{1/2+s}(\CT)}^2
   \lesssim
   \|u - w\|_{H^{1/2+s}(\CT)}^2
   +
   \sum_{l=1}^L h_{l\lag}^{-2s}
   \Bigl(
      \|u - w_{l\lag}\|_{L^2(\gamma_l)}^2
      +
      \|u - w_{l\mor}\|_{L^2(\gamma_l)}^2
   \Bigr)
   \quad\forall w\in X_h
\]
uniformly for $s\in (0,\min\{1/2,r\}]$.
\end{lemma}

\begin{proof}
Let $w\in X_h$ be given. We adapt $w$ such that the new function satisfies
the jump conditions defining $V_h$, cf. (\ref{def_Vh}). We set
\[
   v := w + \sum_{l=1}^L r^l \in X_h
\]
with
\[
   r^l := \left\{
      \begin{array}{ll}
         E_l \pi_l (w_{l\lag}|_{\gamma_l} - w_{l\mor}|_{\gamma_l})
                   & \hbox{on } \bar\G_{l\lag},
         \\
         0         & \hbox{elsewhere.}    
   \end{array}
   \right.
\]
Here, $E_l$ and $\pi_l$ are the extension and projection operators
specified in (\ref{def_El}) and (\ref{def_pil}), respectively.
Note that, since $\pi_l (w_{l\lag}|_{\gamma_l} - w_{l\mor}|_{\gamma_l})$
vanishes at the endpoints of $\gamma_l$, the extension
$E_l \pi_l (w_{l\lag}|_{\gamma_l} - w_{l\mor}|_{\gamma_l})$
vanishes on $\bG_{l\lag}\setminus\gamma_l$.
Therefore, using (\ref{pi_orth}) one obtains
\beas
   \<[v], \mu\>_{\gamma_l}
   &=&
   \<v_{l\lag} - v_{l\mor}, \mu\>_{\gamma_l}
   =
   \<w_{l\lag} + r^l - w_{l\mor}, \mu\>_{\gamma_l}
   \\
   &=&
   \<w_{l\lag} - w_{l\mor}
     +
     \pi_l (w_{l\lag}|_{\gamma_l} - w_{l\mor}|_{\gamma_l}),
   \mu\>_{\gamma_l}
   =
   0
   \qquad \forall\mu\in M_{k,l},\ l=1, \ldots, L.
\eeas
That is, $v\in V_h$. We start bounding the error by
\bea \label{pf8}
   \|u - v\|_{H^{1/2+s}(\CT)}^2
   &=&
   \sum_{i=1}^N\;
   \Bigl\|u_i - w_i - \sum_{l\in\{1,\ldots,L\}:\, l\lag=i} r^l
   \Bigr\|_{H^{1/2+s}(\G_i)}^2
   \nonumber
   \\
   &\lesssim&
   \sum_{i=1}^N \|u_i - w_i\|_{H^{1/2+s}(\G_i)}^2
   +
   \sum_{l=1}^L \|r^l\|_{H^{1/2+s}(\G_{l\lag})}^2.
\eea
Applying Lemma~\ref{la_E}, (\ref{pi_cont}) and the triangle inequality
we find that there holds
\bea \label{pf9}
   \|r^l\|_{H^{1/2+s}(\G_{l\lag})}
   &\lesssim&
   h_{l\lag}^{-s}
   \|\pi_l (w_{l\lag}|_{\gamma_l} - w_{l\mor}|_{\gamma_l})\|
   _{L^2(\gamma_l)} 
   \lesssim
   h_{l\lag}^{-s} \|w_{l\lag} - w_{l\mor}\|_{L^2(\gamma_l)}
   \nonumber
   \\
   &\lesssim&
   h_{l\lag}^{-s}
   \Bigl(
      \|u - w_{l\lag}\|_{L^2(\gamma_l)}
      +
      \|u - w_{l\mor}\|_{L^2(\gamma_l)}
   \Bigr).
\eea
Combining (\ref{pf8}) and (\ref{pf9}) one obtains the assertion.
\end{proof}

The next result provides an a priori error estimate for the mortar BEM.

\begin{theorem} \label{thm_error} 
Let $u$ and $u_h$ be the solutions of {\rm(\ref{VF})} and {\rm(\ref{MBEM})},
respectively. Assuming that $u\in\tilde H^{1/2+r}(\G)$ ($r\in (0,1/2]$)
there holds $\lambda\in \prod_{l=1}^L H^r(\gamma_l)$ and we have the a priori
error estimate
\beas
   \|u-u_h\|_{H^{1/2}(\CT)}^2
   &\lesssim&
      s^{-2}
   |\log\uh|
      \left(
      \|u - v\|_{H^{1/2+s}(\CT)}^2
      +
      \uh^{-2s}
      \sum_{l=1}^L
      \Bigl(
         \|u - v_{l\lag}\|_{L^2(\gamma_l)}^2
         +
         \|u - v_{l\mor}\|_{L^2(\gamma_l)}^2
      \Bigr)
      \right)
\\
&&
      +\
      |\log\uh|^3\, \|\lambda - \mu\|_{L^2(\gamma)}^2
   \qquad \forall v\in X_h,\ \forall \mu\in M_k
\eeas
uniformly for $s\in (0, r]$. Here, $\lambda$ is the Lagrangian multiplier
defined by {\rm(\ref{IP}), (\ref{def_lambda})}.
\end{theorem}

\begin{proof}
Since $u\in \tilde H^{1/2+r}(\G)$ there holds
\[
   \lambda|_{\gamma_l}
   =
   \bt_{l\lag}\cdot(V\bcurlS{\G} u)|_{\gamma_l}
   \in H^r(\gamma_l),
   \quad l=1, \ldots, L.
\]
To this end note that
$\bcurlS{\G}:\; \tilde H^{1/2+r}(\Gamma)\to \bi{\tilde H}^{r-1/2}_\tg(\Gamma)$
(combine (\ref{curl_tilde}) with the continuity
$\bcurlS{\G}:\; H_0^1(\Gamma)\to \bi{L}^2_\tg(\Gamma)$) and
$V:\; \bi{\tilde H}^{r-1/2}_\tg(\Gamma) \to \bi{H}^{1/2+r}_\tg(\Gamma)$.
The trace theorem concludes the claimed regularity of $\lambda$.
In particular there holds $\lambda\in L^2(\gamma)$.

By definition of $V_h$, and making us of Lemma~\ref{la_IP}, we find
\[
   a(u-u_h, w)
   =
   a(u, w) - F(w)
   =
   -b(w, \lambda)
   \quad\forall w\in V_h.
\]
Application of $(\ref{b_cont2})$ yields
\[
   a(u-u_h, w)
   \lesssim
   |\log\uh|\, \inf_{\mu\in M_k} \|\lambda - \mu\|_{L^2(\gamma)} 
   \|\bcurlS{H} w\|_{\bi{\tilde H}^{-1/2}_\tg(\G)}
   \quad\forall w\in V_h.
\]
Therefore, combining Theorem~\ref{thm_Strang} with Lemma~\ref{la_error_Vh}
we obtain
\beas
   \|u-u_h\|_{H^{1/2}(\CT)}^2
   &\lesssim&
   |\log\uh|
   \left\{
      s^{-2}
      \left(
      \|u - v\|_{H^{1/2+s}(\CT)}^2
      +
      \sum_{l=1}^L h_{l\lag}^{-2s}
      \Bigl(
         \|u - v_{l\lag}\|_{L^2(\gamma_l)}^2
         +
         \|u - v_{l\mor}\|_{L^2(\gamma_l)}^2
      \Bigr)
      \right)
\right.
\\
&&
\left.
\qquad\qquad
      +\
      |\log\uh|^2\, \|\lambda - \mu\|_{L^2(\gamma)}^2
   \right\}
   \qquad \forall v\in X_h,\ \forall\mu\in M_k.
\eeas
This proves the statement.
\end{proof}

\noindent{\bf Proof of Theorem~\ref{thm_main}.}
By Theorem~\ref{thm_Strang}, system (\ref{MBEM}) is uniquely solvable. 
We employ the general a priori estimate by Theorem~\ref{thm_error} to show
the given error bound. By standard approximation theory there exist $v\in X_h$
and $\mu\in M_k$ such that
\[
   \|u-v\|_{H^{1/2+s}(\CT)}^2 \lesssim h^{2(r-s)}\|u\|_{\tilde H^{1/2+r}(\G)}^2
   \quad\mbox{and}\quad
   \|\lambda-\mu\|_{L^2(\gamma)}^2
   \lesssim
   k^{2r} \sum_{l=1}^L\|\lambda\|_{H^r(\gamma_l)}^2,
\]
and as in the proof of Theorem~\ref{thm_error} one concludes that
$\sum_{l=1}^L \|\lambda\|_{H^r(\gamma_l)}^2
 \lesssim \|u\|_{\tilde H^{1/2+r}(\G)}^2$.
By Lemma~\ref{la_trace} one bounds
\[
   \|u-v_{l\lag}\|_{L^2(\gamma_l)}
   \lesssim
   s^{-1/2} \|u-v\|_{H^{1/2+s}(\G_{l\lag})}
   \lesssim
   s^{-1/2} h^{r-s}\|u\|_{\tilde H^{1/2+r}(\G)},
\]
and accordingly the mortar part $\|u - v_{l\mor}\|_{L^2(\gamma_l)}$.
Using these bounds in Theorem~\ref{thm_error} and selecting $s=|\log\uh|^{-1}$
one obtains the assertion.
\qed

%%%%%%%%%%%%%%%%%%%%%%%%%%%%%%%%%%%%%%%%%%%%%%%%%%%%%%%%%%%%%%%%%%%%%%%%%%%%%%%
\section{Numerical results} \label{sec_num}
\setcounter{figure}{0}
\setcounter{equation}{0}

We consider the model problem (\ref{VF}) with $\G=(0,1)\times (0,1)$
and $f=1$. In this case there holds $u\in\tilde H^{1/2+r}(\G)$ for
any $r<1/2$ so that by Theorem~\ref{thm_main}
we expect a convergence of the mortar method close to $h^{1/2}$,
the convergence of the conforming BEM, cf. Remark~\ref{rem_error}.
This assumes that the mesh sizes $h$ (of the sub-domain meshes) and
$k$ (of the meshes for the Lagrangian multiplier on the skeleton) are
proportional, which will be the case in all our experiments. In fact,
the elements of the mesh for the Lagrangian multiplier will always consist
of two or three elements of the trace mesh.

Since the exact solution $u$ to (\ref{Wuf}) is unknown
we approximate an upper bound for the semi-norm
$|u-u_h|_{H^{1/2}(\CT)}$.
Here, we follow the strategy from \cite{GaticaHH_BLM}. Let us
recall the procedure and discussion.

By the ellipticity of $V$ and (\ref{pf2}) there holds
\be \label{n1}
   a(u-u_h,u-u_h)
   \gtrsim
   |u-u_h|_{H^{1/2}(\CT)}^2.
\ee
On the other hand, using that $u$ solves (\ref{Wuf}) and $u_h\in V_h$ solves
(\ref{MBEM}), one finds
\[
   a(u-u_h,u-u_h)
   =
   a(u,u) - 2a(u,u_h) + a(u_h,u_h)
   =
   \<Wu, u\>_\G - 2a(u,u_h) + F(u_h).
\]
By (\ref{IP}) there holds
\[
   a(u,u_h)
   =
   F(u_h) - \<[u_h], \lambda\>_\gamma
\]
such that, with the previous relation,
\be \label{n2}
   a(u-u_h,u-u_h)
   =
   \<Wu,u\>_\G - F(u_h) + 2\,\<[u_h], \lambda\>_\gamma
   \le
   \<Wu,u\>_\G - F(u_h) + 2\,\|[u_h]\|_{L^2(\gamma)} \|\lambda\|_{L^2(\gamma)}.
\ee
Like in the proof of Theorem~\ref{thm_error} one sees that
$\|\lambda\|_{L^2(\gamma)}$ is bounded. Therefore, by (\ref{n1}) we find that
\[
   |u-u_h|_{H^{1/2}(\CT)}^2
   \lesssim
   |\<Wu,u\>_\G - F(u_h)| + \|[u_h]\|_{L^2(\gamma)}.
\]
The terms $F(u_h)$ and $\|[u_h]\|_{L^2(\gamma)}$ are directly accessible
and $\<Wu,u\>_\G$ can be approximated by an extrapolated
value that we denote by $\|u\|_{\rm ex}^2$ (cf.~\cite{ErvinHS_93_hpB}).
Therefore, instead of the relative error
\[
   \|u-u_h\|_{H^{1/2}(\CT)}/\|u\|_{H^{1/2}(\G)},
\]
we present results for the expression
\be \label{n3}
   \Bigl(|\,\|u\|_{\rm ex}^2 - F(u_h)| + \|[u_h]\|_{L^2(\gamma)}
   \Bigr)^{1/2}/\|u\|_{\rm ex}
\ee
which is, up to a constant factor, an upper bound for
$|u-u_h|_{H^{1/2}(\CT)}/\|u\|_{\rm ex}$.

In the figures below we show different error curves, indicated by numbers
$(n)$ ($n=1,\ldots,4$) as follows.

\medskip
\begin{tabular}{lll}
(1)   & $\Bigl(|\,\|u\|_{\rm ex}^2 - F(u_h)|
               +
               \, \|[u_h]\|_{L^2(\gamma)}
         \Bigr)^{1/2}$
      & ``mortar BEM''
\\
(2)   & $\Bigl(|\,\|u\|_{\rm ex}^2 - F(u_h)|\Bigr)^{1/2}$
      & ``error1''
\\
(3)   & $\|[u_h]\|_{L^2(\gamma)}^{1/2}$
      & ``error2''
\\
(4)   & $a(u-\tilde u_h, u-\tilde u_h)^{1/2}$
      & ``conforming BEM''
\end{tabular}

\medskip
Here, $\tilde u_h$ denotes a conforming boundary element solution.
Additionally, all curves are normalized by $\|u\|_{\rm ex}$.

Therefore, to resume, an error curve (1) represents the upper bound
(\ref{n3}) for the (normalized) error $|u-u_h|_{H^{1/2}(\CT)}$ of the
mortar BEM. Curves (2) and (3) are the two components of (1).
Here, (3) controls the non-conformity of the mortar approximant $u_h$.
Curve (4) represents the error of the conforming BEM. In this case it is
equivalent to the error in energy norm $\|u-\tilde u_h\|_{\tilde H^{1/2}(\G)}$.

All results are plotted on double logarithmic scales versus $1/h$.
For our numerical experiments we always use rectangular meshes and
in this section, $h_i$ refers to the length of the longest edge on $\G_i$, and
$h:=\max_ih_i$ as before.

\medskip
\noindent{\bf Conforming sub-domain decomposition.}

{\bf Experiment 1} (conforming mesh, results in Figure~\ref{fig_error1_1a}).
First let us consider a conforming decomposition of $\G$ into four
sub-domains as indicated in Figure~\ref{domainmesh}(a).
Moreover, let us first test the case where the separate meshes on
the sub-domains form globally conforming meshes (we take uniform meshes
consisting of squares).
The corresponding results are shown in Figure~\ref{fig_error1_1a}.
Along with the curves (1), (2), (4) we plot the values of $h^{1/2}$.
The numerical results indicate a convergence of the order $O(h^{1/2})$,
for the conforming as well as the mortar BEM. According to the discussion
above this is the best one can expect. The curves (1) and (2),
referring to our upper bound (\ref{n3}) and the first term in
(\ref{n3}), respectively, are almost identical. This means that the
second term in (\ref{n3}), which in the next plots will be labeled
by (3), is negligible in comparison. Indeed, in this symmetric case the jumps
$[u_h]$ disappear and the numerical results vanish at the order of
single precision. Therefore, in this plot, we do not show the curve (3).

We do not observe a logarithmical perturbation of the convergence in
this range of number of unknowns. This may be caused by the fact that
we are not including the $L^2$-parts in the error since our results are,
up to constant factors, upper bounds only for the semi-norm
$|u-u_h|_{H^{1/2}(\CT)}$. Also, we do not know whether our bounds
including the logarithmic terms are sharp.

\begin{figure} [htb]
\begin{center}
\includegraphics[width=0.7\textwidth]{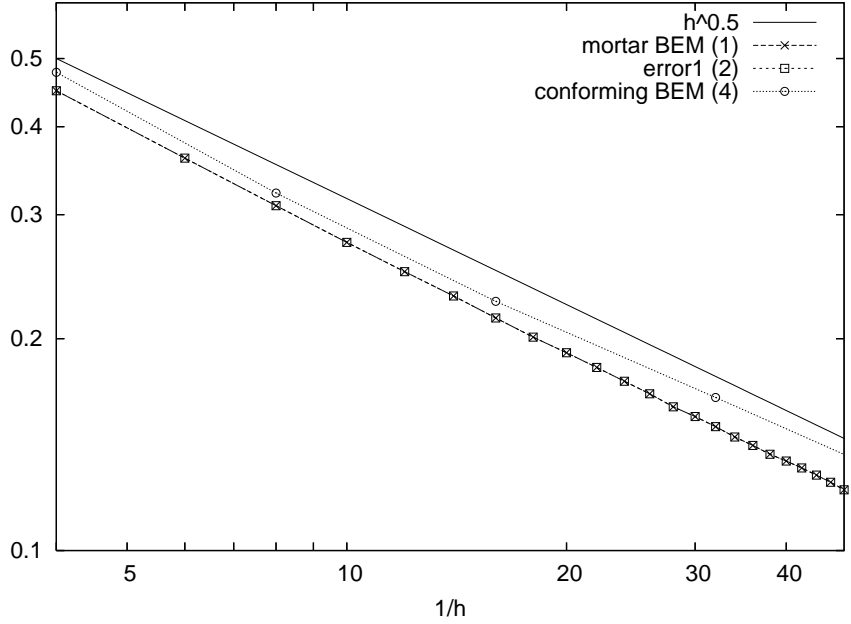}
\end{center}
\caption{Conforming sub-domain decomposition with conforming meshes.}
\label{fig_error1_1a}
\end{figure}

{\bf Experiment 2} (non-conforming mesh, results in
Figure~\ref{fig_errors1_2max}).
Now let us test globally non-conforming meshes. Again we use uniform
meshes consisting of squares on each sub-domain. We mesh as in
Figure~\ref{domainmesh}(a) starting with 2, 3, 4, and 5 ``slides'' on
$\G_1$, $\G_2$, $\G_3$, and $\G_4$, respectively and increase the number
of slides in each sub-domain by one in each step of our sequence of meshes.
The corresponding results are shown in Figure~\ref{fig_errors1_2max}.
Again, a convergence of the expected order $O(h^{1/2})$ is confirmed.
Curve (3) indicates very fast convergence of the jumps
$\|[u_h]\|_{L^2(\gamma)}\to 0$. In the experiments below, however, we
observe a slower convergence. In this particular sequence of meshes,
where we increase the slides on the sub-domains by the same amount,
the trace meshes from different sides on a particular interface edge
approach each other in a certain sense. We conjecture that this specific
situation (``approaching'' conforming meshes) causes the fast convergence
of the jumps.

\begin{figure} [htb]
\begin{center}
\includegraphics[width=0.7\textwidth]{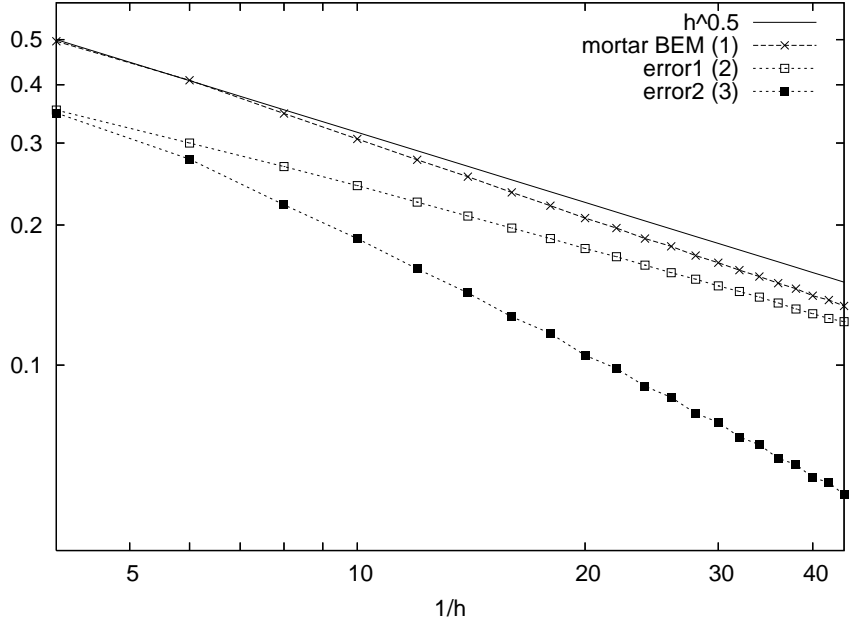}
\end{center}
\caption{Conforming sub-domain decomposition with non-conforming meshes,
         same refinement steps on sub-domains.}
\label{fig_errors1_2max}
\end{figure}

{\bf Experiment 3} (non-conforming mesh, results in
Figure~\ref{fig_errors1_3max}).
For the next experiment we start with a mesh of four squares on each
sub-domain (the sub-domains are again as in Figure~\ref{domainmesh}(a)),
and increase the numbers of slides on 
different sub-domains by different steps (increase by $2$, $3$, $4$, $5$
slides on $\G_1$, $\G_2$, $\G_3$, $\G_4$, respectively).
In this case both error parts, curves (2) and (3), behave like $O(h^{1/2})$,
confirming our a priori error estimate and thus the good performance of the
mortar BEM. Let us note, however, that the part $\|[u_h]\|_{\gamma}^{1/2}$
of the error expression (\ref{n3}) is an overestimation.
Indeed, our substitution (\ref{n3}) for $|u-u_h|_{H^{1/2}(\CT)}/\|u\|_{\rm ex}$ 
is not precise. On the one hand we replaced the term
$2\, \|\lambda\|_{L^2(\gamma)}$ in (\ref{n2}) by $1$
(and the generic constant in (\ref{n1}) by $1$). On the other hand the term 
$\<[u_h], \lambda\>_\gamma$ is of higher order than $\|[u_h]\|_{L^2(\gamma)}$. 
According to (\ref{pf10}) and by standard approximation theory there holds
for any $r<1/2$
\[
   |\<[u_h],\lambda\>_\gamma|
   \lesssim
   \|[u_h]\|_{L^2(\gamma)} \inf_{\psi\in M_k} \|\lambda-\psi\|_{L^2(\gamma)}
   \lesssim
   k^r (\sum_{l=1}^L|\lambda|_{H^r(\gamma_l)}^2)^{1/2}
   \|[u_h]\|_{L^2(\gamma)}.
\]
This shows that $\<[u_h],\lambda\>_\gamma$ is of higher order than
$\|[u_h]\|_{L^2(\gamma)}$. Note that, by the proof of Theorem~\ref{thm_error}
and since $u\in H^{1/2+r}(\G)$, one has the regularity
\(
   \lambda\in \prod_{l=1}^L H^r(\gamma_l)\quad\forall r<1/2.
\)
Therefore, by (\ref{n2}) the term
\[
   \Bigl(|\,\|u\|_{\rm ex}^2 - F(u_h)|\Bigr)^{1/2}/\|u\|_{\rm ex}
   \qquad \mbox{(curve (2),\ ``error1'')}
\]
is asymptotically equal to
\[
   a(u-u_h, u-u_h)^{1/2}/\|u\|_{\rm ex}
\]
and this dominates the error.

\begin{figure} [htb]
\begin{center}
\includegraphics[width=0.7\textwidth]{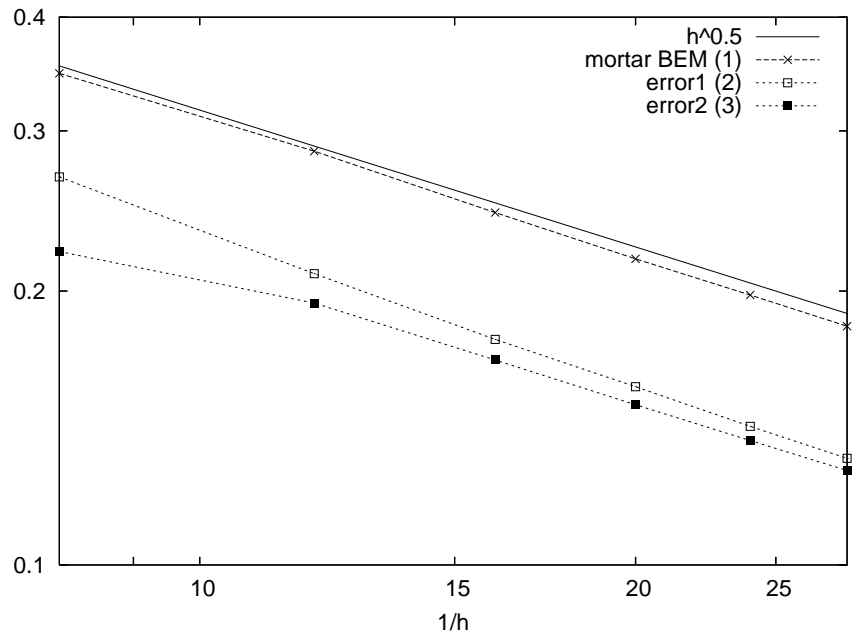}
\end{center}
\caption{Conforming sub-domain decomposition with non-conforming meshes,
         different refinement steps on sub-domains.}
\label{fig_errors1_3max}
\end{figure}

%\medskip
\newpage
\noindent{\bf Non-conforming sub-domain decomposition.}

{\bf Experiment 4} (non-conforming mesh, results in Figure~\ref{fig_error2_3}).
Finally, we consider the fully non-conforming mortar method, i.e. with
non-conforming sub-domain decomposition and non-conforming meshes.
We decompose $\G$ into three sub-domains as in Figure~\ref{domainmesh4} and
use the initial mesh given there on the left.
Then slides on sub-domains are increased
in each direction by 3, 2, 1 on $\G_1$, $\G_2$, $\G_3$, respectively, in
each step. The second mesh is on the right in Figure~\ref{domainmesh4}.
Note that in each second step the cross-point $(0,0)$ between
the sub-domains is a hanging node and our theory includes this case.
The numerical results are shown in Figure~\ref{fig_error2_3}
and again confirm the expected convergence of the mortar BEM.

In this case, the meshes for the Lagrangian multiplier are coarsenings
of the trace meshes from $\G_2$ on $\gamma_{12}$ and $\gamma_{23}$, and
of the trace mesh from $\G_1$ on $\gamma_{13}$. We always join two elements
of the respective trace mesh to form an element of the Lagrangian multiplier
mesh, except for an odd number of elements of the trace mesh when one set
of three elements is joined.
The corresponding numbers of unknowns for the steps are listed in
Table~\ref{tab_4}.

\begin{figure}[htb] 
\begin{center}
\includegraphics[width=0.8\textwidth]{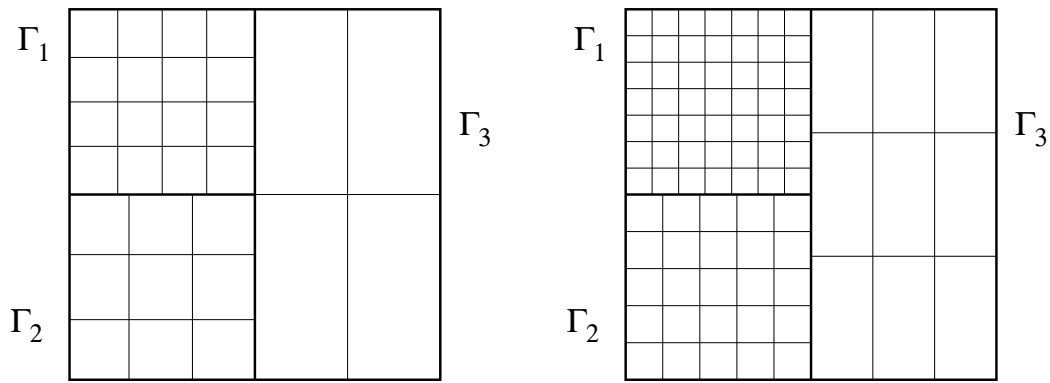}
\end{center}
\caption{Conforming sub-domain decomposition with non-conforming meshes.}
\label{domainmesh4}	
\end{figure}

\begin{figure} [htb]
\begin{center}
\includegraphics[width=0.7\textwidth]{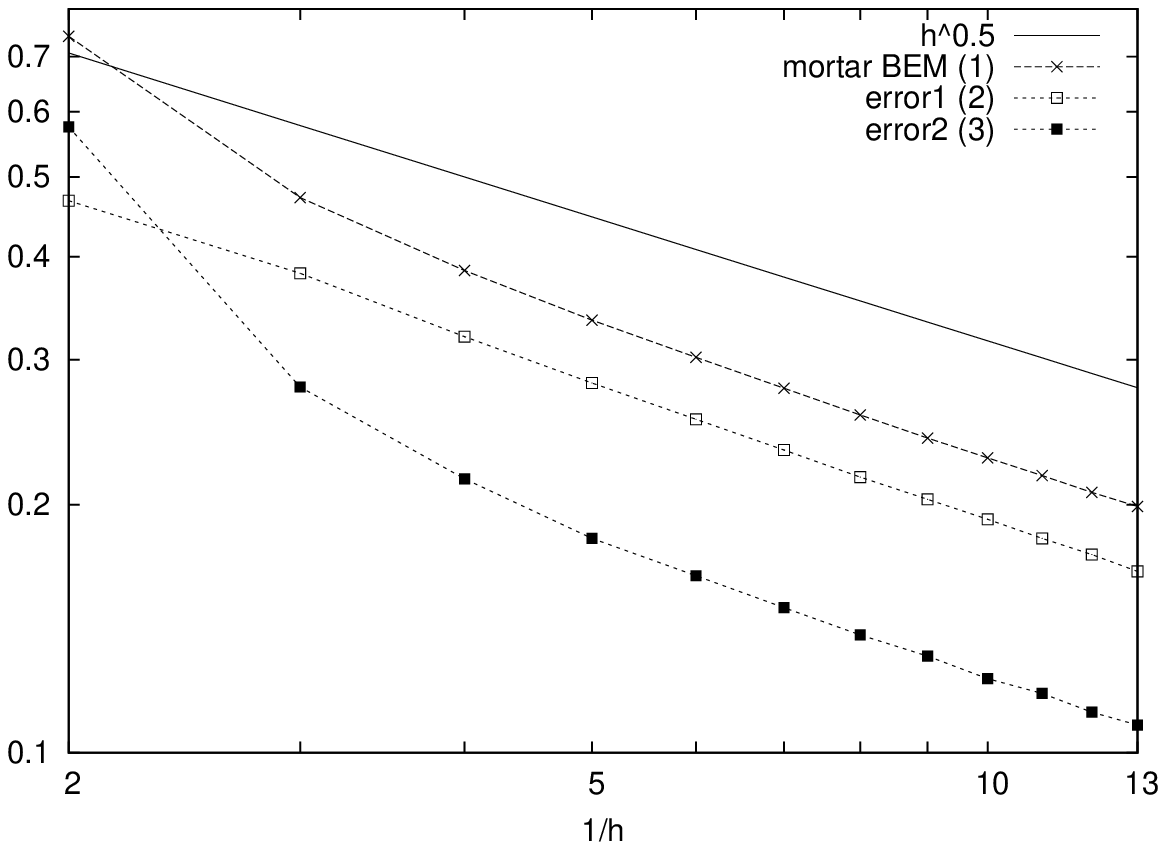}
\end{center}
\caption{Non-conforming sub-domain decomposition with non-conforming meshes,
         different refinement steps on sub-domains.}
\label{fig_error2_3}
\end{figure}

\begin{table}
\begin{center}
\begin{tabular}{r|rr|rr}
$h=h_3$ & $h_1$ & $h_2$ & dim($X_h$) & dim($M_k$)
\\ \hline
  0.5000 & 0.1250 & 0.1667&   27  &     4 \\
  0.3333 & 0.0625 & 0.0833&   80  &     7 \\
  0.2500 & 0.0417 & 0.0556&  161  &    11 \\
  0.2000 & 0.0313 & 0.0417&  270  &    14 \\
  0.1667 & 0.0250 & 0.0333&  407  &    18 \\
  0.1429 & 0.0208 & 0.0278&  572  &    21 \\
  0.1250 & 0.0179 & 0.0238&  765  &    25 \\
  0.1111 & 0.0156 & 0.0208&  986  &    28 \\
  0.1000 & 0.0139 & 0.0185& 1235  &    32 \\
  0.0909 & 0.0125 & 0.0167& 1512  &    35 \\
  0.0833 & 0.0114 & 0.0152& 1817  &    39 \\
  0.0769 & 0.0104 & 0.0139& 2150  &    42 
\end{tabular}
\end{center}
\caption{Dimensions and mesh sizes for experiment 4.}
\label{tab_4}
\end{table}

%%%%%%%%%%%%%%%%%%%%%%%%%%%%%%%%%%%%%%%%%%%%%%%%%%%%%%%%%%%%%%%%%%%%%%%%%%%%%%%
\clearpage
\bibliographystyle{siam}
\bibliography{../../bib/bib,../../bib/heuer,../../bib/fem}

\end{document}